\documentclass{article}
\title{\textbf{Counting edges of different types\\in a local graph of a Grassmann graph}}
\usepackage{authblk}
\author[]{Ian Seong}
\date{}
\usepackage {parskip}
\usepackage [margin=2.5cm] {geometry}
\usepackage{enumitem}
\usepackage[utf8]{inputenc}
\usepackage[T1]{fontenc}
\usepackage{geometry}
\usepackage{amsmath, amssymb, amsthm}
\usepackage{hyperref}
\usepackage{graphicx}
\usepackage{mathtools}
\usepackage{relsize}
\usepackage{bm}
\usepackage{esvect}
\usepackage{circuitikz}
\usepackage{tikz}
\usepackage{hhline}
\usepackage{enumitem}
\usepackage{comment}
\usepackage{ytableau}
\usepackage{pdflscape}
\usepackage{mathabx}
\def\Stab{{\rm{Stab}}}
\newtheoremstyle{dotless}{}{}{\itshape}{}{\bfseries}{}{ }{}
\theoremstyle{dotless}
\newtheorem{theorem}{Theorem}[section]

\newtheorem{lemma}[theorem]{Lemma}

\newtheoremstyle{dotlessdef}{}{}{}{}{\bfseries}{}{ }{}
\theoremstyle{dotlessdef}
\newtheorem{definition}[theorem]{Definition}
\newtheorem{remark}[theorem]{Remark}

\begin{document}
\maketitle
\begin{abstract}
    Let $\mathbb{F}_q$ denote a finite field with $q$ elements. Let $n,k$ denote integers with $n>2k\geq 6$. Let $V$ denote a vector space over $\mathbb{F}_{q}$ that has dimension $n$. The vertex set of the Grassmann graph $J_q(n,k)$ consists of the $k$-dimensional subspaces of $V$. Two vertices of $J_q(n,k)$ are adjacent whenever their intersection has dimension $k-1$. Let $\partial$ denote the path-length distance function of $J_q(n,k)$. Pick a vertex $y$. In this paper we define three types of edges in $X$, namely type $0$, type $+$, and type $-$; for adjacent vertices $w,z$ such that $\partial(w,y)=\partial(z,y)$, the type of the edge $wz$ depends on the subspaces $w+z,w,z,w\cap z$ and their intersections with $y$. Pick a vertex $x$ such that $1<\partial(x,y)<k$. Let $\Gamma(x)$ denote the local graph of $x$ in $J_q(n,k)$.  Our general goal is to count the number of edges in $\Gamma(x)$ for each type. Consider a two-vertex stabilizer $\Stab(x,y)$ in $GL(V)$; it is known that the $\Stab(x,y)$-action on $\Gamma(x)$ has five orbits. Pick two orbits $\mathcal{O},\mathcal{N}$ that are not necessarily distinct; for a given $w\in \mathcal{O}$, we find the number of vertices in $z\in \mathcal{N}$ such that the edge $wz$ has (i) type $0$, (ii) type $+$, (iii) type $-$. To find these numbers, we use many results that involve a projective geometry $P_q(n)$, which is the set of all subspaces of $V$.
\end{abstract}

\section{Introduction}
This paper is about a class of distance-regular graphs \cite{BBIT, BCN} called the Grassmann graphs. These graphs are defined from projective geometries, which we briefly describe. Let $\mathbb{F}_q$ denote a finite field with $q$ elements. Let $n\geq 1$. Let $V$ denote a vector space of dimension $n$ over $\mathbb{F}_q$. Let the set $P=P_q(n)$ consist of the subspaces of $V$. The set $P$, together with inclusion partial order, is a poset called a projective geometry. For $n>k\geq 1$, the vertex set $X$ of the Grassmann graph $J_q(n,k)$ consists of the $k$-dimensional subspaces of $V$. Two vertices $x,y\in X$ are adjacent whenever the subspace $x\cap y$ has dimension $k-1$. The graph $J_q(n,k)$ is known to be distance-regular \cite[Section~9.3]{BCN}. The structure of $J_q(n,k)$ has been widely studied, and there are many open questions related to this graph \cite{GK2025JCT,Lee2020LAA,LIW2020LAA,Metsch}.

Let $\partial$ denote the path-length distance function of $J_q(n,k)$. Pick $y\in X$. For each pair of adjacent vertices $w,z\in X$ such that $\partial(w,y)=\partial(z,y)$, we assign the edge $wz$ with one of three types, namely type $0$, type $+$, and type $-$; as we will see, the type of $wz$ depends on the subspaces $w+z,w,z,w\cap z$ and their intersections with $y$. Pick $x\in X$ such that $1<\partial(x,y)<k$. Let $\Gamma(x)$ denote the set of vertices that are adjacent to $x$. The general goal of this paper is to count the number of edges in $\Gamma(x)$ that have (i) type $0$, (ii) type $+$, (iii) type $-$.

To set the stage, we present some results from \cite{Seong2, Seong3}. Let $u,v\in P$ such that $u\subseteq v$ and $\dim v=\dim u+1$. Write
\begin{equation*}
    u\in P_{i,j}, \qquad \qquad v\in P_{r,s}.
\end{equation*}
By \cite[Lemma~2.3]{Seong3}, either (i) $r=i+1$ and $s=j$, or (ii) $r=i$ and $s=j+1$. We say that $v$ $\slash$-covers $u$ whenever (i) holds, and $v$ $\backslash$-covers $u$ whenever (ii) holds. We now focus on the set $X$. For each pair of adjacent vertices $w,z\in X$ such that $\partial(w,y)=\partial(z,y)$, we say that the edge $wz$ has
\begin{enumerate}[label=(\roman*)]
    \item type $0$ whenever the subspace $w+z$ $/$-covers each of $w$ and $z$, which $\backslash$-covers $w\cap z$;
    \item type $+$ whenever the subspace $w+z$ $\backslash$-covers each of $w$ and $z$;
    \item type $-$ whenever each of the subspaces $w$ and $z$ $/$-covers $w\cap z$.
\end{enumerate}
Let $\Stab(x,y)$ denote the subgroup of $GL(V)$ consisting of the elements that fix both $x$ and $y$. By \cite[Theorem~6.15]{Seong2}, the action of $\Stab(x,y)$ on $\Gamma(x)$ has five orbits:
\begin{align*}
    \mathcal{B}_{xy}&=\{z\in \Gamma(x)\mid \partial(z,y)=\partial(x,y)+1\};\\
    \mathcal{C}_{xy}&=\{z\in \Gamma(x)\mid \partial(z,y)=\partial(x,y)-1\};\\
    \mathcal{A}^{0}_{xy}&=\{z\in \Gamma(x)\mid \text{$\partial(z,y)=\partial(x,y)$ and the edge $zx$ has type $0$}\};\\
    \mathcal{A}^{+}_{xy}&=\{z\in \Gamma(x)\mid \text{$\partial(z,y)=\partial(x,y)$ and the edge $zx$ has type $+$}\};\\
    \mathcal{A}^{-}_{xy}&=\{z\in \Gamma(x)\mid \text{$\partial(z,y)=\partial(x,y)$ and the edge $zx$ has type $-$}\}.
\end{align*}
Furthermore, these orbits form an equitable partition of $\Gamma(x)$. See \cite[Theorem~9.1]{Seong2} for the corresponding structure constants.

We now summarize the main results of this paper. Define diagonal matrices $K_1, K_2 \in \text{Mat}_{P}(\mathbb{C})$ as follows. For $u \in P$ the $(u,u)$-entries are
\begin{equation*}
    \bigl(K_1\bigl)_{u,u}=q^{\frac{k}{2}-i}, \qquad \qquad \bigl(K_2\bigl)_{u,u}=q^{j-\frac{n-k}{2}},
\end{equation*}
where $u\in P_{i,j}$. The matrices $K_1,K_2$ are invertible. Define matrices $L_1,L_2,R_1,R_2\in \text{Mat}_{P}(\mathbb{C})$ as follows. For $u,v\in P$ their $(u,v)$-entries are
\begin{align*}
    &\bigl(L_1\bigr)_{u,v}=\begin{cases}
        1&\text{if }v\text{ $\slash$-covers }u,\\
        0&\text{if }v\text{ does not $\slash$-cover }u,
    \end{cases}
    &\bigl(L_2\bigr)_{u,v}=\begin{cases}
        1&\text{if }v\text{ $\backslash$-covers }u,\\
        0&\text{if }v\text{ does not $\backslash$-cover }u,
    \end{cases}\\
    &\bigl(R_1\bigr)_{u,v}=\begin{cases}
        1&\text{if }u\text{ $\slash$-covers }v,\\
        0&\text{if }u\text{ does not $\slash$-cover }v,
    \end{cases} 
    &\bigl(R_2\bigr)_{u,v}=\begin{cases}
        1&\text{if }u\text{ $\backslash$-covers }v,\\
        0&\text{if }u\text{ does not $\backslash$-cover }v.
    \end{cases}
\end{align*}

Let $\mathcal{H}$ denote the subalgebra of $\text{Mat}_{P}(\mathbb{C})$ generated by $L_1, L_2, R_1, R_2, K_1^{\pm 1}, K_2^{\pm 1}$. We now define five matrices in $\mathcal{H}$ that are related to the five orbits of $\Stab(x,y)$. Define 
\begin{equation*}
    R=L_1R_2,\qquad \qquad \qquad L=L_2R_1.
\end{equation*}
Define matrices $F^{0},F^{+},F^{-}\in \text{Mat}_{P}(\mathbb{C})$ as follows. For $u,v\in P$ their $(u,v)$-entries are
    \begin{align*}
        \bigl(F^{0}\bigr)_{u,v}&=\begin{cases}
            1&\text{if $u+v$ $\slash$-covers each of $u$ and $v$, which $\backslash$-covers $u\cap v$},\\
            0&\text{otherwise},
        \end{cases}\\
        \bigl(F^{+}\bigr)_{u,v}&=\begin{cases}
            1&\text{if $u+v$ $\backslash$-covers each of $u$ and $v$},\\
            0&\text{otherwise},
        \end{cases}\\
        \bigl(F^{-}\bigr)_{u,v}&=\begin{cases}
            1&\text{if each of $u$ and $v$ $\slash$-covers $u\cap v$},\\
            0&\text{otherwise}.
        \end{cases}
    \end{align*}
    Observe that for $w,x\in X$ the $(w,x)$-entries of $R,L,F^{0},F^{+},F^{-}$ are    \begin{align*}
        &R_{w,z}=\begin{cases}
            1&\text{if $\partial(w,y)=\partial(z,y)+1$,}\\
            0&\text{otherwise,}
        \end{cases} &\bigl(F^{0}\bigr)_{w,z}=\begin{cases}
            1&\text{if the edge $wz$ has type $0$,}\\
            0&\text{otherwise,}
        \end{cases}\\
        &L_{w,z}=\begin{cases}
            1&\text{if $\partial(w,y)=\partial(z,y)-1$,}\\
            0&\text{otherwise,}
        \end{cases}
        &\bigl(F^{+}\bigr)_{w,z}=\begin{cases}
            1&\text{if the edge $wz$ has type $+$,}\\
            0&\text{otherwise,}
        \end{cases}\\
        & &\bigl(F^{-}\bigr)_{w,z}=\begin{cases}
            1&\text{if the edge $wz$ has type $-$,}\\
            0&\text{otherwise.}
        \end{cases}
    \end{align*}

    We show that the following (\ref{intro1})--(\ref{intro12}) hold:
        \begin{equation}
        \label{intro1}
            q^2 RF^0-F^0R+\Bigl(q^{\frac{k}{2}}K_1+q^{\frac{n-k}{2}}K_2-(q+1)I\Bigr)R=0;
        \end{equation}
        \begin{equation}
            \label{intro2}
            qRF^+-F^+R-F^0R+(q-1)^{-1}\Bigl(q^{\frac{n}{2}+1}K_1K_2^{-1}-q^{\frac{k}{2}}K_1-q^{\frac{n-k}{2}}K_2+I\Bigr)R=0;
        \end{equation}       
        \begin{equation}
        \label{intro3}
            qRF^- -F^-R-F^0R+(q-1)^{-1}\Bigl(q^{\frac{n}{2}+1}K_1^{-1}K_2-q^{\frac{k}{2}}K_1-q^{\frac{n-k}{2}}K_2+I\Bigr)R=0;
        \end{equation}
        \begin{equation}
        \label{intro4}
            q^2F^0L-LF^0+\Bigl(q^{\frac{k}{2}+1}K_1+q^{\frac{n-k}{2}+1}K_2-(q+1)I\Bigr)L=0;
        \end{equation}
        \begin{equation}
            qF^{+}L-LF^{+}-LF^0+(q-1)^{-1}\Bigl(q^{\frac{n}{2}+1}K_1K_2^{-1}-q^{\frac{k}{2}+1}K_1-q^{\frac{n-k}{2}+1}K_2+I\Bigr)L=0;
        \end{equation}
        \begin{equation}
            qF^{-}L-LF^{-}-LF^0+(q-1)^{-1}\Bigl(q^{\frac{n}{2}+1}K_1^{-1}K_2-q^{\frac{k}{2}+1}K_1-q^{\frac{n-k}{2}+1}K_2+I\Bigr)L=0;
        \end{equation}
        \begin{equation}
        \label{intro11}
            \begin{aligned}
            0=&\;LR-qF^+F^- -(q-1)^{-1}\Bigl(\bigl(q^{\frac{n}{2}+1}K_1^{-1}K_2-q^{\frac{n-k}{2}+1}K_2\bigr)F^{+}+\bigl(q^{\frac{n}{2}+1}K_1K_2^{-1}-q^{\frac{k}{2}+1}K_1\bigr)F^{-}\Bigr)\\
            &\;\qquad \qquad \qquad -(q-1)^{-2}\Bigl(q^{\frac{n}{2}+1}K_1K_2-q^{n-\frac{k}{2}+1}K_1-q^{\frac{n+k}{2}+1}K_2+q^{n+1}I\Bigr);
        \end{aligned}
        \end{equation}
        \begin{equation}
        \label{intro12}
        \begin{aligned}
            0=&\;qRL-\bigl(F^{0}\bigr)^2-F^{0}F^{+}-F^{0}F^{-}-F^{+}F^{-}\\
            &\;\qquad -(q-1)^{-1}\Bigl(\bigl(q^{\frac{k}{2}}K_1+q^{\frac{n-k}{2}}K_2-2I\bigr)F^{0}+\bigl(q^{\frac{k}{2}}K_1-I\bigr)F^{+}+\bigl(q^{\frac{n-k}{2}}K_2-I\bigr)F^{-}\Bigr)\\
            &\;\qquad \qquad-(q-1)^{-2}\bigl(q^{\frac{n}{2}}K_1K_2-q^{\frac{k}{2}}K_1-q^{\frac{n-k}{2}}K_2+I\bigr).
        \end{aligned}
    \end{equation}

    Define $\mathcal{A}_{xy}=\mathcal{A}^{0}_{xy}\cup \mathcal{A}^{+}_{xy}\cup \mathcal{A}^{-}_{xy}$. For $w\in \mathcal{A}_{xy}$, we find the $(w,x)$-entries of many matrices that involve $F^{0}, F^{+}, F^{-}$. In what follows, we use the notation
\begin{equation*}
    [m]=\frac{q^{m}-1}{q-1}\qquad \qquad \qquad (m\in \mathbb{Z}).
\end{equation*}
Referring to the table below, for each matrix $M$ in the header column and each orbit $\mathcal{O}$ in the header row, the $(M,\mathcal{O})$-entry gives the $(w,x)$-entry of $M$ when $w\in \mathcal{O}$. Write $i=\partial(x,y)$.

    \begin{center}
            \begin{tabular}{c|c c c}
                & $\mathcal{A}^{0}_{xy}$ & $\mathcal{A}^{+}_{xy}$ & $\mathcal{A}^{-}_{xy}$ \\
                \hline \\
                $\bigl(F^{0}\bigr)^2$ & $2q^i-q-2$ & $0$ & $0$ \\ \\
                $F^0F^+$ & $0$ & $(q-1)[i]$ & $0$ \\ \\
                $F^0F^-$ & $0$ & $0$ & $(q-1)[i]$ \\ \\
                $F^+F^0$ & $0$ & $(q-1)[i]$ & $0$ \\ \\
                $\bigl(F^+\bigr)^2$ & $q^{i+1}[n-k-i]$ & $q[n-k]-q^i-q$ & $0$ \\ \\
                $F^+F^-$ & $0$ & $0$ & $0$\\ \\
                $F^-F^0$ & $0$ & $0$ & $(q-1)[i]$ \\ \\
                $F^-F^+$ & $0$ & $0$ & $0$ \\ \\
                $\bigl(F^-\bigr)^2$ & $q^{i+1}[k-i]$ & $0$ & $q[k]-q^i-q$
            \end{tabular}
        \end{center}
        
    Referring to the table below, for each orbit $\mathcal{O}$ in the header column, and each orbit $\mathcal{N}$ in the header row, the $(\mathcal{O},\mathcal{N})$-entry gives a $3\times 1$-matrix that satisfies the following: each entry of the matrix corresponds to the number of vertices in $\mathcal{N}$ that share an edge of type $0$, $+$, $-$ respectively with a given vertex in $\mathcal{O}$. We omit the entries $(0,0,0)^{\top}$, where $\top$ is the matrix-transpose. Write $i=\partial(x,y)$. 
    \begin{center}
\begin{tabular}{c|c c c c c}
     & $\mathcal{B}_{xy}$ & $\mathcal{C}_{xy}$ & $\mathcal{A}_{xy}^{0}$ & $\mathcal{A}_{xy}^{+}$ & $\mathcal{A}_{xy}^{-}$\\
    \hline \\
    $\mathcal{B}_{xy}$ & $\begin{pmatrix}
        \substack{2q^{i+1}-q-1}\\
        \substack{q^{i+2}[n-k-i-1]}\\
        \substack{q^{i+2}[k-i-1]}
    \end{pmatrix}$ &  &  &  & \\
   \\
    $\mathcal{C}_{xy}$ &  & $\begin{pmatrix}
        0\\
        q[i-1]\\
        q[i-1]
    \end{pmatrix}$ &  &  & \\
   \\
    $\mathcal{A}_{xy}^{0}$ &  &  & $\begin{pmatrix}
        2q^{i}-q-2\\0\\0
    \end{pmatrix}$ & $\begin{pmatrix}
        0\\q^{i+1}[n-k-i]\\0
    \end{pmatrix}$ & $\begin{pmatrix}
        0\\0\\q^{i+1}[k-i]
    \end{pmatrix}$ \\
    \\
    $\mathcal{A}_{xy}^{+}$ & & & $\begin{pmatrix}
        0\\(q-1)[i]\\0
    \end{pmatrix}$ & $\begin{pmatrix}
       (q-1)[i]\\ q[n-k]-q^i-q \\0
    \end{pmatrix}$ & \\
    \\
    $\mathcal{A}_{xy}^{-}$ & & & $\begin{pmatrix}
        0\\0\\(q-1)[i]
    \end{pmatrix}$ & & $\begin{pmatrix}
       (q-1)[i] \\0\\ q[k]-q^i-q
    \end{pmatrix}$\\
\end{tabular}
\end{center}
The above table is the main result of this paper. We finish the paper with an appendix that contains relations between the generators of $\mathcal{H}$.

This paper is organized as follows. In Section \ref{prelim} we discuss some preliminaries on the projective geometry $P$. In Section \ref{H} we discuss the algebra $\mathcal{H}$ and its center. In Section \ref{somerelations} we present some relations that involve the matrices $R,L,F^{0},F^{+},F^{-}$. In Section \ref{grassmann} we discuss the Grassmann graph $J_q(n,k)$ and the two-vertex stabilizer $\Stab(x,y)$. In Section \ref{types}, for $\mathcal{O}\in \{\mathcal{B}_{xy},\mathcal{C}_{xy}\}$, a given vertex $w\in \mathcal{O}$, and a $\Stab(x,y)$-orbit $\mathcal{N}$, we count the number of vertices $z\in\mathcal{N}$ adjacent to $w$ such that the edge $wz$ has (i) type $0$, (ii) type $+$, (iii) type $-$. In Section \ref{entry}, for $w\in \mathcal{A}_{xy}$, we find the $(w,x)$-entries of many matrices that involve $F^{0}, F^{+}, F^{-}$. In Section \ref{typeII}, for $\mathcal{O}\in \{\mathcal{A}^{0}_{xy}, \mathcal{A}^{+}_{xy}, \mathcal{A}^{-}_{xy}\}$, a given vertex $w\in \mathcal{O}$, and a $\Stab(x,y)$-orbit $\mathcal{N}$, we count the number of vertices $z\in\mathcal{N}$ adjacent to $w$ such that the edge $wz$ has (i) type $0$, (ii) type $+$, (iii) type $-$. Section $\ref{appendix}$ is the appendix on the relations between the generators of $\mathcal{H}$.
\section{Projective geometry $P_q(n)$}
\label{prelim}

We now begin our formal argument. Let $\mathbb{F}_q$ denote a finite field with $q$ elements. Let $n\geq 1$. Let $V$ denote a vector space of dimension $n$ over $\mathbb{F}_q$. Let the set $P=P_q(n)$ consist of the subspaces of $V$. The set $P$, together with inclusion partial order, is a poset called a \emph{projective geometry}. For $u,v\in P$, we say that \emph{$v$ covers $u$} whenever $u\subseteq v$ and $\dim v=\dim u+1$. Let $\text{Mat}_{P}(\mathbb{C})$ denote the $\mathbb{C}$-algebra consisting of the matrices with rows and columns indexed by $P$ and all entries in $\mathbb{C}$. 

For $0\leq \ell \leq n$, let the set $P_{\ell}$ consist of elements of $P$ that have dimension $\ell$. We have a partition
\begin{equation}
\label{partition}
    P=\bigcup_{\ell=0}^n P_{\ell}.
\end{equation}

We now refine this partition. Let $k$ denote an integer such that $n>k\geq 1$. Pick $y\in P_k$. For $0\leq i\leq k$ and $0\leq j\leq n-k$, define
\begin{equation*}
    P_{i,j}=\{u\in P\mid \dim (u\cap y)=i,\;\dim u=i+j\}.
\end{equation*}
In the diagram below, we illustrate the set $P$ and the subsets $P_{i,j}$ ($0\leq i\leq k$, $0\leq j\leq n-k$).

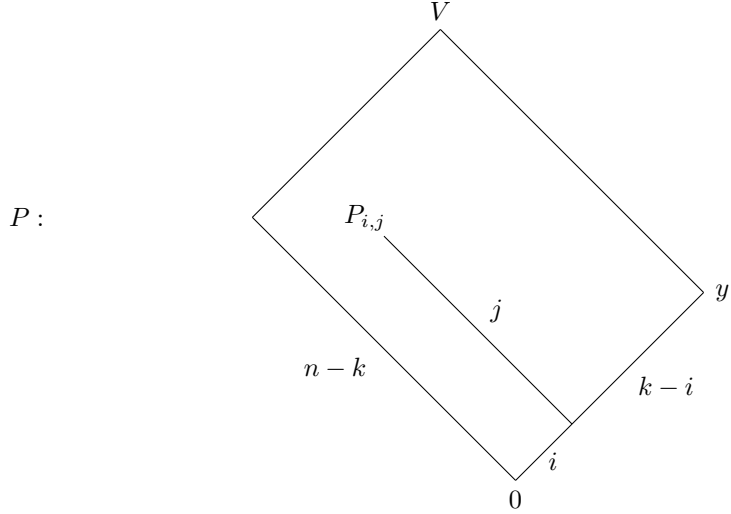
\begin{figure}[!ht]
\centering
{%
\begin{circuitikz}
\tikzstyle{every node}=[font=\normalsize]
\draw [short] (2.5,-3.5) -- (-0.9,-0.1);
\draw [short] (-0.9,-0.1) -- (1.6,2.4);
\draw [short] (1.6,2.4) -- (5,-1);
\draw [short] (5,-1) -- (2.5,-3.5);
\draw [short] (0.75,-0.25) -- (3.25,-2.75);
\node [font=\normalsize] at (1.6,2.65) {$V$};
\node [font=\normalsize] at (5.25,-1) {$y$};
\node [font=\normalsize] at (2.5,-3.75) {$0$};
\node [font=\normalsize] at (3,-3.25) {$i$};
\node [font=\normalsize] at (2.25,-1.25) {$j$};
\node [font=\normalsize] at (0.5,0) {$P_{i,j}$};
\node [font=\normalsize] at (4.5,-2.25) {$k-i$};
\node [font=\normalsize] at (0.1,-2) {$n-k$};
\node [font=\normalsize] at (-4,0) {$P:$};
\node [font=\normalsize] at (8,0) {$\;$};
\end{circuitikz}
}%
\caption{The projective geometry $P$ and the location of $P_{i,j}$.}
\label{projective}
\end{figure}
\newpage
    
    Note that for $0\leq \ell\leq n$,
    \begin{equation*}
        P_{\ell}=\bigcup_{i,j}P_{i,j},
    \end{equation*}
    where the union is over the ordered pairs $(i,j)$ such that $0\leq i\leq k$ and $0\leq j\leq n-k$ and $i+j=\ell$. 
    
    Earlier, we described the covering relation on $P$. We now give a refinement of this covering relation. 

\begin{lemma}{\rm{\cite[Lemma~2.3]{Seong3}}}
    \label{coverlem}
        Let $u,v\in P$ such that $v$ covers $u$. Write
        \begin{equation*}
            u\in P_{i,j}, \qquad \qquad v\in P_{r,s}.
        \end{equation*}
        Then either (i) $r=i+1$ and $s=j$, or (ii) $r=i$ and $s=j+1$.
    \end{lemma}
        
    \begin{definition}
    \label{slashcover}
        Referring to Lemma \ref{coverlem},  we say that \emph{$v$ $\slash$-covers $u$} whenever (i) holds, and \emph{$v$ $\backslash$-covers $u$} whenever (ii) holds.
    \end{definition}

    We illustrate Definition \ref{slashcover} using the diagrams below.
    \begin{figure}[!ht]
\centering
{%
\begin{circuitikz}
\tikzstyle{every node}=[font=\normalsize]
\draw [short] (2.5,-3.5) -- (-0.9,-0.1);
\draw [short] (-0.9,-0.1) -- (1.6,2.4);
\draw [short] (1.6,2.4) -- (5,-1);
\draw [short] (5,-1) -- (2.5,-3.5);
\draw [short] (0.9,-0.1) -- (1.1,0.1);
\node [font=\normalsize] at (1.6,2.65) {$V$};
\node [font=\normalsize] at (5.25,-1) {$y$};
\node [font=\normalsize] at (2.5,-3.75) {$0$};
\node [font=\normalsize] at (1.25,0.25) {$v$};
\node [font=\normalsize] at (0.75,-0.25) {$u$};
\node [font=\normalsize] at (4,-2.5) {$k$};
\node [font=\normalsize] at (0.1,-2) {$n-k$};
\node [font=\normalsize] at (2,-4.55) {Figure 2A: $v$ $\slash$-covers $u$.};
\end{circuitikz}
\hspace{0.25cm}
\begin{circuitikz}
\draw [short] (9.5,-3.5) -- (6.1,-0.1);
\draw [short] (6.1,-0.1) -- (8.6,2.4);
\draw [short] (8.6,2.4) -- (12,-1);
\draw [short] (12,-1) -- (9.5,-3.5);
\draw [short] (8.4,0.1) -- (8.6,-0.1);

\node [font=\normalsize] at (8.6,2.65) {$V$};
\node [font=\normalsize] at (12.25,-1) {$y$};
\node [font=\normalsize] at (9.5,-3.75) {$0$};
\node [font=\normalsize] at (8.25,0.25) {$v$};
\node [font=\normalsize] at (8.75,-0.25) {$u$};
\node [font=\normalsize] at (11,-2.5) {$k$};
\node [font=\normalsize] at (7.1,-2) {$n-k$};
\node [font=\normalsize] at (9,-4.55) {Figure 2B: $v$ $\backslash$-covers $u$.};
\end{circuitikz}
}%
\end{figure}

\renewcommand{\thefigure}{\arabic{figure}}

\setcounter{figure}{2}
\newpage

In what follows, we use the notation
    \begin{equation*}
        [m]=\frac{q^m-1}{q-1}\qquad \qquad (m\in \mathbb{Z}).
    \end{equation*}

\begin{lemma}{\rm{\cite[Lemma~5.1]{Watanabe}}}
    \label{cover}
        For $0\leq i\leq k$ and $0\leq j\leq n-k$, the following (i)--(iv) hold:
        \begin{enumerate}[label=(\roman*)]
            \item each element in $P_{i,j}$ $\slash$-covers exactly  $q^j[i]$ elements;

            \item each element in $P_{i,j}$ $\backslash$-covers exactly $[j]$ elements;

            \item each element in $P_{i,j}$ is $\slash$-covered by exactly $[k-i]$ elements;

            \item each element in $P_{i,j}$ is $\backslash$-covered by exactly $q^{k-i}[n-k-j]$ elements.
        \end{enumerate}
    \end{lemma}

    We briefly comment on the symmetry of $P$. Recall that $GL(V)$ consists of invertible $\mathbb{F}_q$-linear maps from $V$ to $V$. The action of $GL(V)$ on $V$ induces a permutation action of $GL(V)$ on $P$. The action preserves dimension and the inclusion partial order. 

    \section{The algebra $\mathcal{H}$ and its center}
    \label{H}

    In this section we recall from \cite[Section~7]{Watanabe} the subalgebra $\mathcal{H}$ of $\text{Mat}_{P}(\mathbb{C})$. We recall from \cite[Definition~3.1,~5.3]{Seong3} some useful matrices in $\mathcal{H}$, namely $R,L,F^0, F^+, F^-$. We recall from \cite[Theorem~4.1]{Seong3} a generating set for the center of $\mathcal{H}$.

    Throughout this paper we use the following notation. For $u,v\in P$ and $M\in \text{Mat}_{P}(\mathbb{C})$, let $M_{u,v}$ denote the $(u,v)$-entry of $M$.

    \begin{definition}
    \label{kdef}
        We define diagonal matrices $K_1, K_2 \in \text{Mat}_{P}(\mathbb{C})$ as follows. For $u \in P$ their $(u,u)$-entries are
        \begin{equation*}
            \bigl(K_1\bigl)_{u,u}=q^{\frac{k}{2}-i}, \qquad \qquad \bigl(K_2\bigl)_{u,u}=q^{j-\frac{n-k}{2}},
        \end{equation*}
        where $u\in P_{i,j}$. Note that $K_1,K_2$ are invertible.   
    \end{definition}

    \begin{definition}
    \label{l1l2r1r2def}
        We define matrices $L_1,L_2,R_1,R_2\in \text{Mat}_{P}(\mathbb{C})$ as follows. For $u,v\in P$ their $(u,v)$-entries are
    \begin{align*}
        &\bigl(L_1\bigr)_{u,v}=\begin{cases}
            1&\text{if }v\text{ $\slash$-covers }u,\\
            0&\text{if }v\text{ does not $\slash$-cover }u,
        \end{cases}
        &\bigl(L_2\bigr)_{u,v}=\begin{cases}
            1&\text{if }v\text{ $\backslash$-covers }u,\\
            0&\text{if }v\text{ does not $\backslash$-cover }u,
        \end{cases}\\
        &\bigl(R_1\bigr)_{u,v}=\begin{cases}
            1&\text{if }u\text{ $\slash$-covers }v,\\
            0&\text{if }u\text{ does not $\slash$-cover }v,
        \end{cases} 
        &\bigl(R_2\bigr)_{u,v}=\begin{cases}
            1&\text{if }u\text{ $\backslash$-covers }v,\\
            0&\text{if }u\text{ does not $\backslash$-cover }v.
        \end{cases}
    \end{align*}
    
    Note that $R_1=L_1^{\top}$ and $R_2=L_2^{\top}$, where $\top$ is the matrix-transpose. 
    \end{definition}

    \begin{definition}
    \label{hdef}
        Let $\mathcal{H}$ denote the subalgebra of $\text{Mat}_{P}(\mathbb{C})$ generated by $L_1, L_2, R_1, R_2, K_1^{\pm 1}, K_2^{\pm 1}$. 
    \end{definition}

    Next we define matrices $R,L,F^0, F^+, F^-$. 

    \begin{definition}
    \label{rlDef}
        Define
        \begin{equation}
        \label{rldef}
            R=L_1R_2,\qquad \qquad \qquad L=R_1L_2.
        \end{equation}
        By Lemma \ref{appendix2}(i),(ii) in the appendix, 
        \begin{equation}
        \label{rldef2}
            R=R_2L_1,\qquad \qquad \qquad L=L_2R_1.
        \end{equation}
        Note that $R,L\in \mathcal{H}$. Also note that $R=L^{\top}$.
    \end{definition}

     \begin{definition}
    \label{ftildedef}
        Define matrices $F^{0},F^{+},F^{-}\in \text{Mat}_{P}(\mathbb{C})$ as follows. For $u,v\in P$ their $(u,v)$-entries are
        \begin{align*}
            \bigl(F^{0}\bigr)_{u,v}&=\begin{cases}
                1&\text{if $u+v$ $/$-covers each of $u$ and $v$, which $\backslash$-covers $u\cap v$},\\
                0&\text{otherwise},
            \end{cases}\\
            \bigl(F^{+}\bigr)_{u,v}&=\begin{cases}
                1&\text{if $u+v$ $\backslash$-covers each of $u$ and $v$},\\
                0&\text{otherwise},
            \end{cases}\\
            \bigl(F^{-}\bigr)_{u,v}&=\begin{cases}
                1&\text{if each of $u$ and $v$ $/$-covers $u\cap v$},\\
                0&\text{otherwise}.
            \end{cases}
        \end{align*}
    \end{definition}

    We illustrate the definitions of $F^0, F^+, F^-$ using the diagrams below.

    \begin{figure}[!ht]
\centering
{%
\begin{circuitikz}
\tikzstyle{every node}=[font=\normalsize]
\draw [short] (2.5,-3.5) -- (-1,0);
\draw [short] (-1,0) -- (1.5,2.5);
\draw [short] (1.5,2.5) -- (5,-1);
\draw [short] (5,-1) -- (2.5,-3.5);
\draw [color=red, short] (2.3,-0.8) -- (2.5,-0.6);
\draw [color=red, short] (2.3,-1.2) -- (2.5,-1.4);
\node [font=\normalsize] at (1.5,2.75) {$V$};
\node [font=\normalsize] at (5.25,-1) {$y$};
\node [font=\normalsize] at (2.5,-3.75) {$0$};
\node [font=\normalsize] at (2.6,-0.4) {$u+v$};
\node [font=\normalsize] at (2,-1) {$u,v$};
\node [font=\normalsize] at (2.6,-1.6) {$u\cap v$};
\node [font=\normalsize] at (4,-2.5) {$k$};
\node [font=\normalsize] at (0.1,-2) {$n-k$};
\node [font=\normalsize] at (2,-4.75) {Figure 3A: $(F^0)_{u,v}=1$.};
\end{circuitikz}
\hspace{0.25cm}
\begin{circuitikz}
\tikzstyle{every node}=[font=\normalsize]
\draw [short] (2.5,-3.5) -- (-1,0);
\draw [short] (-1,0) -- (1.5,2.5);
\draw [short] (1.5,2.5) -- (5,-1);
\draw [short] (5,-1) -- (2.5,-3.5);
\draw [color=red, short] (1.7,-0.8) -- (1.5,-0.6);
\draw [color=red, short] (2.3,-1.2) -- (2.5,-1.4);
\node [font=\normalsize] at (1.5,2.75) {$V$};
\node [font=\normalsize] at (5.25,-1) {$y$};
\node [font=\normalsize] at (2.5,-3.75) {$0$};
\node [font=\normalsize] at (1.4,-0.4) {$u+v$};
\node [font=\normalsize] at (2,-1) {$u,v$};
\node [font=\normalsize] at (2.6,-1.6) {$u\cap v$};
\node [font=\normalsize] at (4,-2.5) {$k$};
\node [font=\normalsize] at (0.1,-2) {$n-k$};
\node [font=\normalsize] at (2,-4.75) {Figure 3B: $(F^+)_{u,v}=1$.};
\end{circuitikz}
}%
\end{figure}
\begin{figure}[!ht]
\centering
{%
\begin{circuitikz}
\tikzstyle{every node}=[font=\normalsize]
\draw [short] (2.5,-3.5) -- (-1,0);
\draw [short] (-1,0) -- (1.5,2.5);
\draw [short] (1.5,2.5) -- (5,-1);
\draw [short] (5,-1) -- (2.5,-3.5);
\draw [color=red, short] (2.3,-0.8) -- (2.5,-0.6);
\draw [color=red, short] (1.7,-1.2) -- (1.5,-1.4);
\node [font=\normalsize] at (1.5,2.75) {$V$};
\node [font=\normalsize] at (5.25,-1) {$y$};
\node [font=\normalsize] at (2.5,-3.75) {$0$};
\node [font=\normalsize] at (2.6,-0.4) {$u+v$};
\node [font=\normalsize] at (2,-1) {$u,v$};
\node [font=\normalsize] at (1.4,-1.6) {$u\cap v$};
\node [font=\normalsize] at (4,-2.5) {$k$};
\node [font=\normalsize] at (0.1,-2) {$n-k$};
\node [font=\normalsize] at (2,-4.75) {Figure 3C: $(F^-)_{u,v}=1$.};
\end{circuitikz}
}%
\end{figure}

\begin{definition}
\label{fadef}
    Define
    \begin{equation}
    \label{fdef}
        F=F^{0}+F^{+}+F^{-}.
    \end{equation}
    
\end{definition}

\begin{lemma}
        The matrices $F^{0}, F^{+}, F^{-}, F$ are contained in $\mathcal{H}$.
    \end{lemma}
    \begin{proof}
        For the matrices $F^{0}, F^{+}, F^{-}$ see \cite[Lemma~3.5]{Seong3}. By construction, $F$ is also contained in $\mathcal{H}$.
    \end{proof}

Next we express $F^{0}, F^{+}, F^{-}$ in terms of the generators of $\mathcal{H}$.
\begin{lemma}{\rm{\cite[Lemma~3.2--3.4]{Seong3}}}
    \label{f0}
        The matrices $F^{0}, F^{+}, F^{-}$ satisfy
        \begin{align}
            F^{0}&=L_1R_1-R_1L_1+(q-1)^{-1}\Bigl(q^{\frac{n}{2}}K_1^{-1}K_2-q^{\frac{k}{2}}K_1-q^{\frac{n-k}{2}}K_2+I\Bigr)\label{f01}\\
            &=R_2L_2-L_2R_2+(q-1)^{-1}\Bigl(q^{\frac{n}{2}}K_1K_2^{-1}-q^{\frac{k}{2}}K_1-q^{\frac{n-k}{2}}K_2+I\Bigr)\label{f02},\\
            F^{+}&=L_2R_2-q^{\frac{k}{2}}(q-1)^{-1}K_1\Bigl(q^{\frac{n-k}{2}}K_2^{-1}-I\Bigr),\label{f+}\\
            F^{-}&=R_1L_1-q^{\frac{n-k}{2}}(q-1)^{-1}\Bigl(q^{\frac{k}{2}}K_1^{-1}-I\Bigr)K_2.\label{f-}
        \end{align}
    \end{lemma}
     
    Next we recall the generating set for the center of $\mathcal{H}$.

    \begin{lemma}{\rm{\cite[Theorem~4.1]{Seong3}}}
\label{central}
    The center of $\mathcal{H}$ is generated by the following three elements:
    \begin{align}
        \Omega_0&=q^{-\frac{n}{2}}\Bigl((q-1)F^{0}K_1^{-1}K_2^{-1}+q^{\frac{n-k}{2}}K_1^{-1}+q^{\frac{k}{2}}K_2^{-1}-K_1^{-1}K_2^{-1}\Bigr),\label{omega0}\\
        \Omega_1&=q^{-\frac{n-k}{2}}\Biggl(qF^{0}K_2^{-1}+(q-1)F^{-}K_2^{-1}+\frac{q^{\frac{k}{2}+1}K_1K_2^{-1}+q^{\frac{n}{2}+1}K_1^{-1}-qK_2^{-1}}{q-1}\Biggr)-\frac{q}{q-1}I,\label{omega1}\\ 
        \Omega_2&=q^{-\frac{k}{2}}\Biggl(qF^0K_1^{-1}+(q-1)F^{+}K_1^{-1}+\frac{q^{\frac{n-k}{2}+1}K_1^{-1}K_2+q^{\frac{n}{2}+1}K_2^{-1}-qK_1^{-1}}{q-1}\Biggr)-\frac{q}{q-1}I.\label{omega2}
    \end{align}
    \end{lemma}

    \begin{lemma}{\rm{\cite[Lemma~4.3]{Seong3}}}
\label{fcentral}
    We have
    \begin{align}
        F^{0}&=(q-1)^{-1}\Bigl(q^{\frac{n}{2}}\Omega_0K_1K_2-q^{\frac{k}{2}}K_1-q^{\frac{n-k}{2}}K_2+I\Bigr),\label{f0center}\\
        F^{+}&=(q-1)^{-1}\Biggl(q^{\frac{k}{2}}\Omega_2-(q-1)^{-1}\biggl(q^{\frac{n}{2}+1}\Bigl(\Omega_0K_2+K_2^{-1}\Bigr)-2q^{\frac{k}{2}+1}I\biggr)\Biggr)K_1,\label{f+center}\\
        F^{-}&=(q-1)^{-1}\Biggl(q^{\frac{n-k}{2}}\Omega_1-(q-1)^{-1}\biggl( q^{\frac{n}{2}+1}\Bigl(\Omega_0K_1+K_1^{-1}\Bigr)-2q^{\frac{n-k}{2}+1}I\biggr)\Biggr)K_2.\label{f-center}
    \end{align}
\end{lemma}

\begin{lemma}
    \label{commute}
        The matrices $F^{0}, F^{+}, F^{-}, F$ mutually commute.
    \end{lemma}

    \begin{proof}
        The matrices $F^{0},F^{+},F^{-}$ mutually commute by \cite[Lemma~4.5]{Seong3}. By construction, the matrix $F$ commutes with each of $F^{0}, F^{+}, F^{-}$.
    \end{proof}

    \section{Some relations that involve $R,L,F^0, F^+, F^-$}
    \label{somerelations}
    In this section we present some useful relations that involve the matrices $R,L,F^{0}, F^{+}, F^{-}$.
    
    \begin{theorem}
        We have
        \begin{equation}
        \label{RF0}
            q^2 RF^0-F^0R+\Bigl(q^{\frac{k}{2}}K_1+q^{\frac{n-k}{2}}K_2-(q+1)I\Bigr)R=0.
        \end{equation}
    \end{theorem}

    \begin{proof}
        Use (\ref{f0center}) and Lemma \ref{appendix1}(i),(iv),(v),(viii) in the appendix to obtain
        \begin{equation}
        \label{RF0=}
            RF^0=(q-1)^{-1}\Bigl(q^{\frac{n}{2}-2}\Omega_0K_1K_2-q^{\frac{k}{2}-1}K_1-q^{\frac{n-k}{2}-1}K_2+I\Bigr)R.
        \end{equation}
        Evaluate the left-hand side of (\ref{RF0}) using (\ref{f0center}), (\ref{RF0=}). The result follows.
        \end{proof}

    \begin{theorem}
        We have
        \begin{equation}
            \label{RF+1}
            qRF^+-F^+R-F^0R+(q-1)^{-1}\Bigl(q^{\frac{n}{2}+1}K_1K_2^{-1}-q^{\frac{k}{2}}K_1-q^{\frac{n-k}{2}}K_2+I\Bigr)R=0.
        \end{equation}
    \end{theorem}

    \begin{proof}
        Use (\ref{f+center}) and Lemma \ref{appendix1}(i),(iv),(v),(viii) in the appendix to obtain
        \begin{equation}
        \label{RF+=}
            RF^+=(q-1)^{-1}\Biggl(q^{\frac{k}{2}-1}\Omega_2-(q-1)^{-1}\biggl(q^{\frac{n}{2}-1}\Bigl(\Omega_0K_2+q^2K_2^{-1}\Bigr)-2q^{\frac{k}{2}}I\biggr)\Biggr)K_1R.
        \end{equation}
        Evaluate the right-hand side of (\ref{RF+1}) using (\ref{f0center}), (\ref{f+center}), (\ref{RF+=}). The result follows.
    \end{proof}

    \begin{theorem}
        We have
        \begin{equation}
        \label{RF-1}
            qRF^- -F^-R-F^0R+(q-1)^{-1}\Bigl(q^{\frac{n}{2}+1}K_1^{-1}K_2-q^{\frac{k}{2}}K_1-q^{\frac{n-k}{2}}K_2+I\Bigr)R=0.
        \end{equation}
    \end{theorem}
    
    \begin{proof}
        Use (\ref{f-center}) and Lemma \ref{appendix1}(i),(iv),(v),(viii) in the appendix to obtain
        \begin{equation}
        \label{RF-=}
            RF^-=(q-1)^{-1}\Biggl(q^{\frac{n-k}{2}-1}\Omega_1-(q-1)^{-1}\biggl(q^{\frac{n}{2}-1}\Bigl(\Omega_0K_1+q^2K_1^{-1}\Bigr)-2q^{\frac{n-k}{2}}I\biggr)\Biggr)K_2R.
        \end{equation}
        Evaluate the right-hand side of (\ref{RF-1}) using (\ref{f0center}), (\ref{f-center}), (\ref{RF-=}). The result follows.
    \end{proof}

    \begin{theorem}
        We have
        \begin{equation}
        \label{LF0}
            q^2F^0L-LF^0+\Bigl(q^{\frac{k}{2}+1}K_1+q^{\frac{n-k}{2}+1}K_2-(q+1)I\Bigr)L=0.
        \end{equation}
    \end{theorem}

    \begin{proof}
        Use (\ref{f0center}) and Lemma \ref{appendix1}(ii),(iii),(vi),(vii) in the appendix to obtain
        \begin{equation}
        \label{LF0=}
            LF^0=(q-1)^{-1}\Bigl(q^{\frac{n}{2}+2}\Omega_0K_1K_2-q^{\frac{k}{2}+1}K_1-q^{\frac{n-k}{2}+1}K_2+I\Bigr)L.
        \end{equation}
        Evaluate the left-hand side of (\ref{LF0}) using (\ref{f0center}), (\ref{LF0=}). The result follows.
    \end{proof}

    \begin{theorem}
        We have
        \begin{align}
        \label{LF+1}
            qF^{+}L-LF^{+}-LF^0+(q-1)^{-1}\Bigl(q^{\frac{n}{2}+1}K_1K_2^{-1}-q^{\frac{k}{2}+1}K_1-q^{\frac{n-k}{2}+1}K_2+I\Bigr)L=0.
        \end{align}
    \end{theorem}

    \begin{proof}
        Use (\ref{f+center}) and Lemma \ref{appendix1}(ii),(iii),(vi),(vii) in the appendix to obtain
        \begin{equation}
        \label{LF+=}
            LF^+=(q-1)^{-1}\Biggl(q^{\frac{k}{2}+1}\Omega_2-(q-1)^{-1}\biggl(q^{\frac{n}{2}+1}\Bigl(q^2\Omega_0K_2+K_2^{-1}\Bigr)-2q^{\frac{k}{2}+2}I\biggr)\Biggr)K_1L.
        \end{equation}
        Evaluate the right-hand side of (\ref{LF+1}) using (\ref{f0center}), (\ref{f+center}), (\ref{LF+=}). The result follows.
    \end{proof}

    \begin{theorem}
        We have
        \begin{equation}
        \label{LF-1}
            qF^{-}L-LF^{-}-LF^0+(q-1)^{-1}\Bigl(q^{\frac{n}{2}+1}K_1^{-1}K_2-q^{\frac{k}{2}+1}K_1-q^{\frac{n-k}{2}+1}K_2+I\Bigr)L=0.
        \end{equation}
    \end{theorem}
    
    \begin{proof}
        Use (\ref{f-center}) and Lemma \ref{appendix1}(ii),(iii),(vi),(vii) in the appendix to obtain
        \begin{equation}
        \label{LF-=}
            LF^-=(q-1)^{-1}\Biggl(q^{\frac{n-k}{2}+1}\Omega_1-(q-1)^{-1}\biggl(q^{\frac{n}{2}+1}\Bigl(q^2\Omega_0K_1+K_1^{-1}\Bigr)-2q^{\frac{n-k}{2}+2}I\biggr)\Biggr)K_2L.
        \end{equation}
        Evaluate the right-hand side of (\ref{LF-1}) using (\ref{f0center}), (\ref{f-center}), (\ref{LF-=}).
    \end{proof}

    \begin{theorem}
        We have
        \begin{equation}
        \label{LR}
            \begin{aligned}
            0=&\;LR-qF^+F^- -(q-1)^{-1}\Bigl(\bigl(q^{\frac{n}{2}+1}K_1^{-1}K_2-q^{\frac{n-k}{2}+1}K_2\bigr)F^{+}+\bigl(q^{\frac{n}{2}+1}K_1K_2^{-1}-q^{\frac{k}{2}+1}K_1\bigr)F^{-}\Bigr)\\
            &\;\qquad \qquad \qquad -(q-1)^{-2}\Bigl(q^{\frac{n}{2}+1}K_1K_2-q^{n-\frac{k}{2}+1}K_1-q^{\frac{n+k}{2}+1}K_2+q^{n+1}I\Bigr).
        \end{aligned}
        \end{equation}
    \end{theorem}
\begin{proof}
    Evaluate the right-hand side of (\ref{LR}) using (\ref{f+}), (\ref{f-}); we routinely obtain
    \begin{equation}
    \label{LR-}
        LR-qR_1L_1L_2R_2.
    \end{equation}
    Evaluate (\ref{LR-}) using (\ref{rldef}) and Lemma \ref{appendix2}(iii) in the appendix. The result follows.
\end{proof}

\begin{theorem}
    We have
    \begin{equation}
    \label{RL}
        \begin{aligned}
            0=&\;qRL-FF^{0}-F^{+}F^{-}-(q-1)^{-1}\Bigl(\bigl(q^{\frac{k}{2}}K_1+q^{\frac{n-k}{2}}K_2-2I\bigr)F^{0}+\bigl(q^{\frac{k}{2}}K_1-I\bigr)F^{+}+\bigl(q^{\frac{n-k}{2}}K_2-I\bigr)F^{-}\Bigr)\\
            &\;\qquad \qquad-(q-1)^{-2}\bigl(q^{\frac{n}{2}}K_1K_2-q^{\frac{k}{2}}K_1-q^{\frac{n-k}{2}}K_2+I\bigr).
        \end{aligned}
    \end{equation}
\end{theorem}

\begin{proof}
    Evaluate the right-hand side of (\ref{RL}) using (\ref{fdef}), (\ref{f01})--(\ref{f-}); we routinely obtain
    \begin{equation}
    \label{RL=}
        qRL-L_1R_1R_2L_2.
    \end{equation}
    Evaluate (\ref{RL=}) using (\ref{rldef}) and Lemma \ref{appendix2}(iv) in the appendix. The result follows.
\end{proof}
    
\section{The Grassmann graph $J_q(n,k)$ and a two-vertex stabilizer $\Stab(x,y)$}
\label{grassmann}
In this section we define the Grassmann graph $J_q(n,k)$ from the projective geometry $P$. For distinct vertices $x,y$ of $J_q(n,k)$ we define a two-vertex stabilizer $\Stab(x,y)$ in $GL(V)$. We recall from \cite{Seong2} the orbits of the $\Stab(x,y)$-action on the local graph of $x$. We recall the corresponding structure constants.

Let $\Gamma=(X,E)$ denote a graph with no loops or multiple edges, vertex set $X$, edge set $E$, and path-length distance function $\partial$. By the \emph{diameter} $D$ of $\Gamma$, we mean the value $\max\{\partial(x,y)\mid x,y\in X\}$. For a vertex $x\in X$, let the set $\Gamma(x)$ consist of vertices in $X$ that are adjacent to $x$. By the \emph{local graph of $x$}, we mean the subgraph of $\Gamma$ induced on $\Gamma(x)$.

We say that $\Gamma$ is \textit{regular with valency $\kappa$} whenever $|\Gamma(x)|=\kappa$ for all $x\in X$. We say that $\Gamma$ is \emph{distance-regular} whenever for all $0\leq h,i,j\leq D$ and all $x,y\in X$ such that $\partial(x,y)=h$, the cardinality of the set $$\{z\in X\mid\partial(x,z)=i,\; \partial(y,z)=j\}$$ depends only on $h,i,j$. This cardinality is denoted by $p_{i,j}^{h}$. Define 
\begin{equation*}
        b_i=p^{i}_{1,i+1}\; \;  (0\leq i<D),\qquad \qquad  c_i=p^{i}_{1,i-1}\; \; (0<i\leq D), \qquad \qquad a_i=p^{i}_{1,i} \; \; (0\leq i\leq D).
\end{equation*}

We call $b_i,a_i,c_i$ the \emph{intersection numbers of $\Gamma$}.

If $\Gamma$ is distance-regular, then $\Gamma$ is regular with valency $\kappa=b_0$. From now on, we assume that $\Gamma$ is distance-regular.

We now define the Grassmann graph $J_q(n,k)$. Recall the projective geometry $P$ and its partition in (\ref{partition}). For $n>k\geq 1$, the vertex set of $J_q(n,k)$ is $X=P_{k}$. Two vertices $x,y\in X$ are adjacent whenever each of $x,y$ covers $x\cap y$. The graph $J_q(n,k)$ is known to be distance-regular \cite[Section~9.3]{BCN}. By \cite[p.~268]{BCN}, the graph $J_q(n,k)$ is isomorphic to $J_q(n,n-k)$. Without loss, we assume that $n\geq 2k$. The case $n=2k$ is somewhat special, so for the rest of the paper, we assume that $n>2k$. Under this assumption, the diameter of $J_q(n,k)$ is $k$. (See \cite[Theorem~9.3.3]{BCN}.) From now on, we assume that $\Gamma=J_q(n,k)$ with $k\geq 3$.

By \cite[Theorem~9.3.3]{BCN}, the intersection numbers of $\Gamma$ are
\begin{equation}
    \label{sizebc}
    b_i=q^{2i+1}[k-i][n-k-i],\qquad \qquad c_i=[i]^2\qquad \qquad (0\leq i\leq k).
\end{equation}

\begin{definition}
    Pick $y\in X$. For adjacent $w,z\in X$ such that $\partial(w,y)=\partial(z,y)$, we say that the edge $wz$ has
    \begin{enumerate}[label=(\roman*)]
        \item type $0$ whenever the subspace $w+z$ $/$-covers each of $w$ and $z$, which $\backslash$-covers $w\cap z$;
        \item type $+$ whenever the subspace $w+z$ $\backslash$-covers each of $w$ and $z$;
        \item type $-$ whenever each of the subspaces $w$ and $z$ $/$-covers $w\cap z$.
    \end{enumerate} 
\end{definition}

Recall the group $GL(V)$. By \cite[Lemma~3.9]{Seong}, the action of $GL(V)$ on $X$ preserves the path-length distance $\partial$.

\begin{definition}
    For distinct $x,y\in X$ let $\Stab(x,y)$ denote the subgroup of $GL(V)$ consisting of the elements that fix both $x$ and $y$. We call $\Stab(x,y)$ the \emph{stabilizer of $x$ and $y$}.
\end{definition}

Pick distinct $x,y\in X$. Our next goal is to describe the orbits of the $\Stab(x,y)$-action on $\Gamma(x)$. The cases $\partial(x,y)=1$ and $\partial(x,y)=k$ are somewhat special. For the rest of the paper, we assume that $1<\partial(x,y)<k$.

\begin{definition}
    For $x,y\in X$ such that $1<\partial(x,y)<k$, define
    \begin{align*}
        \mathcal{B}_{xy}&=\{z\in \Gamma(x)\mid \partial(z,y)=\partial(x,y)+1\},\\
        \mathcal{C}_{xy}&=\{z\in \Gamma(x)\mid \partial(z,y)=\partial(x,y)-1\},\\
        \mathcal{A}_{xy}&=\{z\in \Gamma(x)\mid \partial(z,y)=\partial(x,y)\}.
    \end{align*}
\end{definition}

\begin{definition}
    For $x,y\in X$ such that $1<\partial(x,y)<k$, define
    \begin{align*}
        \mathcal{A}^{0}_{xy}&=\{z\in \mathcal{A}_{xy}\mid \text{the edge $xz$ has type $0$}\},\\
        \mathcal{A}^{+}_{xy}&=\{z\in \mathcal{A}_{xy}\mid \text{the edge $xz$ has type $+$}\},\\
        \mathcal{A}^{-}_{xy}&=\{z\in \mathcal{A}_{xy}\mid \text{the edge $xz$ has type $-$}\}.
    \end{align*}
\end{definition}
\begin{lemma}{\rm{\cite[Lemma~6.7]{Seong2}}}
    For $x,y\in X$ such that $1<\partial(x,y)<k$, the set $\mathcal{A}_{xy}$ is the disjoint union of the sets $\mathcal{A}_{xy}^{0},\mathcal{A}_{xy}^{+},\mathcal{A}_{xy}^{-}$.
\end{lemma}

\begin{lemma}{\rm{\cite[Theorem~6.15]{Seong2}}}
\label{5orbitslem}
    For $x,y\in X$ such that $1<\partial(x,y)<k$, the following sets are orbits of the $\Stab(x,y)$-action on $\Gamma(x)$:
    \begin{equation}
    \label{5orbits}
        \mathcal{B}_{xy}, \qquad \qquad \mathcal{C}_{xy}, \qquad \qquad \mathcal{A}_{xy}^{0}, \qquad \qquad \mathcal{A}_{xy}^{+}, \qquad \qquad \mathcal{A}_{xy}^{-}.
    \end{equation}
    Furthermore, these orbits form a partition of $\Gamma(x)$.
\end{lemma}

Next we give the cardinality of each orbit in (\ref{5orbits}).

\begin{lemma}{\rm{\cite[Section~9]{Seong}}}
    Pick $x,y\in X$ such that $1<\partial(x,y)<k$. The cardinalities of the sets $\mathcal{B}_{xy}, \mathcal{C}_{xy}$ are
    \begin{equation*}
        \vert \mathcal{B}_{xy}\vert=b_i, \qquad \qquad \vert \mathcal{C}_{xy}\vert=c_i,
    \end{equation*}
    where $i=\partial(x,y)$ and $b_i,c_i$ are from (\ref{sizebc}).
\end{lemma}

\begin{lemma}{\rm{\cite[Lemma~6.8]{Seong2}}}
\label{size}
    For $1<i<k$, let $a_i^{0},a_i^{+},a_i^{-}$ denote the cardinalities of  $\mathcal{A}^{0}_{xy}$, $\mathcal{A}^{+}_{xy}$, $\mathcal{A}^{-}_{xy}$ respectively, where $x,y\in X$ such that $\partial(x,y)=i$. Then
    \begin{equation*}
        a_i^{0}=(q-1)[i]^2, \qquad \qquad a_i^{+}=q^{i+1}[i][n-k-i], \qquad \qquad a_i^{-}=q^{i+1}[i][k-i].
    \end{equation*}
\end{lemma}

By \cite[Definition~6.16]{Seong2}, the orbits in (\ref{5orbits}) form an equitable partition of $\Gamma(x)$ in the sense of \cite[p.~159]{Schwenk}. We present the corresponding structure constants.

\begin{lemma}{\rm{\cite[Theorem~9.1]{Seong2}}}
\label{structure}
For $x,y\in X$ such that $1<\partial(x,y)<k$, consider the orbits of the $\Stab(x,y)$-action on $\Gamma(x)$. Referring to the table below, for each orbit $\mathcal{O}$ in the header column, and each orbit $\mathcal{N}$ in the header row, the $(\mathcal{O},\mathcal{N})$-entry gives the number of vertices in $\mathcal{N}$ that are adjacent to a given vertex in $\mathcal{O}$. Write $i=\partial(x,y)$.

\begin{center}
\begin{tabular}{c|c c c c c}
     & $\mathcal{B}_{xy}$ & $\mathcal{C}_{xy}$ & $\mathcal{A}_{xy}^{0}$ & $\mathcal{A}_{xy}^{+}$ & $\mathcal{A}_{xy}^{-}$\\
    \hline \\
    $\mathcal{B}_{xy}$ & $\substack{q^{i+1}[k-i]\\+q^{i+1}[n-k-i]-q-1}$ & $0$ & $0$ & $q[i]$ & $q[i]$\\
   \\
    $\mathcal{C}_{xy}$ & $0$ & $2q[i-1]$ & $2q^i-q-1$ & $q^{i+1}[n-k-i]$ & $q^{i+1}[k-i]$\\
   \\
    $\mathcal{A}_{xy}^{0}$ & $0$ & $2[i]-1$ & $2q^i-q-2$ & $q^{i+1}[n-k-i]$ &$q^{i+1}[k-i]$\\
    \\
    $\mathcal{A}_{xy}^{+}$ & $q^{i+1}[k-i]$ & $[i]$ & $(q-1)[i]$ & $q[n-k]-q-1$ & $0$\\
    \\
    $\mathcal{A}_{xy}^{-}$ & $q^{i+1}[n-k-i]$ & $[i]$ & $(q-1)[i]$ & $0$ & $q[k]-q-1$\\ \\
\end{tabular}
\end{center}
\end{lemma}

\section{Counting edges of different types in $\Gamma(x)$}
\label{types}

Pick a pair of orbits $\mathcal{O}, \mathcal{N}$ from (\ref{5orbits}). Pick $w\in \mathcal{O}$. Our next goal is to find the number of vertices $z\in \mathcal{N}$ adjacent to $w$ such that the edge $wz$ has (i) type $0$, (ii) type $+$, (iii) type $-$. In this section, we assume that $\mathcal{O}\in \{\mathcal{B}_{xy}, \mathcal{C}_{xy}\}$. The other orbits will be considered in Section \ref{typeII}. 

We first make some observations on the matrices $R,L,F,F^{0},F^{+},F^{-}$ from Definitions \ref{rlDef}--\ref{fadef}.

\begin{remark}
\label{remark2}
    Pick $y\in X$. For a pair of adjacent vertices $w,z\in X$, observe that the $(w,z)$-entries of $R$, $L$, $F$, $F^{0}$, $F^{+}$, $F^{-}$ are
    \begin{align*}
        &R_{w,z}=\begin{cases}
            1&\text{if $\partial(w,y)=\partial(z,y)+1$,}\\
            0&\text{otherwise;}
        \end{cases} &\bigl(F^{0}\bigr)_{w,z}=\begin{cases}
            1&\text{if the edge $wz$ has type $0$,}\\
            0&\text{otherwise;}
        \end{cases}\\
        &L_{w,z}=\begin{cases}
            1&\text{if $\partial(w,y)=\partial(z,y)-1$,}\\
            0&\text{otherwise;}
        \end{cases}
        &\bigl(F^{+}\bigr)_{w,z}=\begin{cases}
            1&\text{if the edge $wz$ has type $+$,}\\
            0&\text{otherwise;}
        \end{cases}\\
        &F_{w,z}=\begin{cases}
            1&\text{if $\partial(w,y)=\partial(z,y)$}\\
            0&\text{otherwise;}
            \end{cases}
            &\bigl(F^{-}\bigr)_{w,z}=\begin{cases}
            1&\text{if the edge $wz$ has type $-$,}\\
            0&\text{otherwise.}
        \end{cases}
    \end{align*}
\end{remark}

In what follows, we use the following notation. For sets $\mathcal{R}\subseteq \mathcal{S}$, denote by $\mathcal{S}\setminus \mathcal{R}$ the complement of $\mathcal{R}$ in $\mathcal{S}$. 

\begin{lemma}
    Pick $x,y\in X$ such that $1<\partial(x,y)<k$. For $w\in \mathcal{B}_{xy}$ and $z\in \Gamma(x)\setminus \mathcal{B}_{xy}$ that are adjacent, the edge $wz$ does not have type $0$, $+$, nor $-$.
\end{lemma}

\begin{proof}
    Observe that $\mathcal{B}_{xy}$ is the only orbit where its vertex $w$ satisfies $\partial(w,y)=\partial(x,y)+1$. Therefore, $\partial(w,y)\neq\partial(z,y)$. The result follows.
\end{proof}

\begin{theorem}
\label{I1}
    Pick $x,y\in X$ such that $1<\partial(x,y)<k$. Pick $w\in \mathcal{B}_{xy}$. In the table below, we display in the right column the number of vertices $z\in \mathcal{B}_{xy}$ adjacent to $w$ such that the edge $wz$ has the type given in the left column. Write $i=\partial(x,y)$.
    
    \begin{center}
    \begin{tabular}{c|c}
        Type & Number of vertices in $\mathcal{B}_{xy}$\\
        \hline \\
        $0$ & $2q^{i+1}-q-1$\\ \\ 
        $+$ & $q^{i+2}[n-k-i-1]$\\ \\
        $-$ & $q^{i+2}[k-i-1]$
    \end{tabular}
    \end{center}
\end{theorem}

\begin{proof}
    We first consider type $0$. We evaluate the $(w,x)$-entry of each term on the left-hand side of (\ref{RF0}).
    
    We claim that 
    \begin{equation}
    \label{RF0wx}
        \bigl(RF^{0}\bigr)_{w,x}=0.
    \end{equation}
    Observe that $\bigl(RF^{0}\bigr)_{w,x}$ counts the number of vertices $z\in \mathcal{A}^{0}_{xy}$ that are adjacent to $w$. Referring to the table given in Lemma \ref{structure}, the ($\mathcal{B}_{xy}, \mathcal{A}^{0}_{xy}$)-entry is equal to $0$. The result follows. 

    Next consider 
    \begin{equation}
    \label{w-entry}
        \Bigl(q^{\frac{k}{2}}K_1+q^{\frac{n-k}{2}}K_2-(q+1)I\Bigr)R.
    \end{equation}
    Note that $R_{w,x}=1$. By Definition \ref{kdef} and the fact that $w\in P_{k-i-1,i+1}$, the $(w,x)$-entry of (\ref{w-entry}) is equal to $2q^{i+1}-q-1$. By the above comments and (\ref{RF0}), the $(w,x)$-entry of $F^{0}R$ is
    \begin{equation}
        \label{F0Rwx}
        \bigl(F^{0}R\bigr)_{w,x}=2q^{i+1}-q-1.
    \end{equation}
    Observe that $\bigl(F^{0}R\bigr)_{w,x}$ counts the number of vertices $z\in \mathcal{B}_{xy}$ adjacent to $w$ such that the edge $wz$ has type $0$. The result follows.

    We have proved the result for type $0$. Next we consider type $+$. We evaluate the $(w,x)$-entry of each term on the right-hand side of (\ref{RF+1}). Observe that $\bigl(RF^{+}\bigr)_{w,x}$ counts the number of vertices $z\in \mathcal{A}^{+}_{xy}$ that are adjacent to $w$. Using the $(\mathcal{B}_{xy},\mathcal{A}^{+}_{xy})$-entry of the table in Lemma \ref{structure}, we get 
    \begin{equation*}
        \bigl(qRF^{+}\bigr)_{w,x}=q^2[i].
    \end{equation*}
    Next consider 
    \begin{equation}
    \label{w-entry2}
        (q-1)^{-1}\Bigl(q^{\frac{n}{2}+1}K_1K_2^{-1}-q^{\frac{k}{2}}K_1-q^{\frac{n-k}{2}}K_2+I\Bigr)R.
    \end{equation}
    Recall that $R_{w,x}=1$. By Definition \ref{kdef} and the fact that $w\in P_{k-i-1,i+1}$, the $(w,x)$-entry of (\ref{w-entry2}) is equal to 
    \begin{equation*}
        q^{i+1}[n-k-i]-[i+1].
    \end{equation*}
    By the above comments and (\ref{F0Rwx}), the $(w,x)$-entry of $F^{+}R$ is equal to  
    \begin{equation*}
        \bigl(F^+R\bigr)_{w,x}=q^{i+2}[n-k-i-1].
    \end{equation*}
    Observe that $\bigl(F^+R\bigr)_{w,x}$ counts the number of vertices $z\in \mathcal{B}_{xy}$ adjacent to $w$ such that the edge $wz$ has type $+$. The result follows.

    We have proved the result for type $+$. The proof for type $-$ is similar, and omitted.
    \end{proof}

    \begin{lemma}
    Pick $x,y\in X$ such that $1<\partial(x,y)<k$. For $w\in \mathcal{C}_{xy}$ and $z\in \Gamma(x)\setminus \mathcal{C}_{xy}$ that are adjacent, the edge $wz$ does not have type $0$, $+$, nor $-$.
\end{lemma}

\begin{proof}
    Observe that $\mathcal{C}_{xy}$ is the only orbit where its vertex $w$ satisfies $\partial(w,y)=\partial(x,y)-1$. Therefore,  $\partial(w,y)\neq \partial(z,y)$. The result follows.
\end{proof}

    \begin{theorem}
\label{I2}
    Pick $x,y\in X$ such that $1<\partial(x,y)<k$. Pick $w\in \mathcal{C}_{xy}$. In the table below, we display in the right column the number of vertices $z\in \mathcal{C}_{xy}$ adjacent to $w$ such that the edge $wz$ has the type given in the left column. Write $i=\partial(x,y)$.
    
    \begin{center}
    \begin{tabular}{c|c}
        Type & Number of vertices in $\mathcal{C}_{xy}$\\
        \hline \\
        $0$ & $0$\\ \\ 
        $+$ & $q[i-1]$\\ \\
        $-$ & $q[i-1]$
    \end{tabular}
    \end{center}
\end{theorem}

\begin{proof}
    We first consider type $0$. We evaluate the $(w,x)$-entry of each term on the left-hand side of (\ref{LF0}).
    
    Observe that $\bigl(LF^{0}\bigr)_{w,x}$ counts the number of vertices $z\in \mathcal{A}^{0}_{xy}$ that are adjacent to $w$. Referring to the table given in Lemma \ref{structure}, the $(\mathcal{C}_{xy},\mathcal{A}^{0}_{xy})$-entry is equal to $2q^{i}-q-1$. Hence the $(w,x)$-entry of $LF^{0}$ is equal to
    \begin{equation*}
        \bigl(LF^{0}\bigr)_{w,x}=2q^i-q-1.
    \end{equation*}
    Next consider 
    \begin{equation}
    \label{w-entry3}
        \Bigl(q^{\frac{k}{2}+1}K_1+q^{\frac{n-k}{2}+1}K_2-(q+1)I\Bigr)L.
    \end{equation}
    Note that $L_{w,x}=1$. By Definition \ref{kdef} and the fact that $w\in P_{k-i+1,i-1}$, the $(w,x)$-entry of (\ref{w-entry3}) is equal to $2q^{i}-q-1$. By the above comments and (\ref{LF0}), we have
    \begin{equation}
        \label{F0Lwx}
        \bigl(F^{0}L\bigr)_{w,x}=0.
    \end{equation}
    Observe that $\bigl(F^{0}L\bigr)_{w,x}$ counts the number of vertices $z\in \mathcal{C}_{xy}$ adjacent to $w$ such that the edge $wz$ has type $0$. The result follows.

    We have proved the result for type $0$. Next we consider type $+$. We evaluate the $(w,x)$-entry of each term on the right-hand side of (\ref{LF+1}). Observe that $\bigl(LF^{+}\bigr)_{w,x}$ counts the number of vertices $z\in \mathcal{A}^{+}_{xy}$ that are adjacent to $w$. Using the $(\mathcal{C}_{xy},\mathcal{A}^{+}_{xy})$-entry of the table in Lemma \ref{structure}, we get 
    \begin{equation*}
        \bigl(LF^{+}\bigr)_{w,x}=q^{i+1}[n-k-i].
    \end{equation*}
    Next consider 
    \begin{equation}
    \label{w-entry4}
        (q-1)^{-1}\Bigl(q^{\frac{n}{2}+1}K_2^{-1}K_1-q^{\frac{k}{2}+1}K_1-q^{\frac{n-k}{2}+1}K_2+I\Bigr)L.
    \end{equation}
    Recall that $L_{w,x}=1$. By Definition \ref{kdef} and the fact that $w\in P_{k-i+1,i-1}$, the $(w,x)$-entry of (\ref{w-entry4}) is equal to 
    \begin{equation*}
        q^i[n-k-i+1]-[i].
    \end{equation*}
    By the above comments and (\ref{F0Lwx}), we have
    \begin{equation*}
        \bigl(F^+L\bigr)_{w,x}=q[i-1].
    \end{equation*}
    Observe that $\bigl(F^+L\bigr)_{w,x}$ counts the number of vertices $z\in \mathcal{C}_{xy}$ adjacent to $w$ such that the edge $wz$ has type $+$. The result follows.

    We have proved the result for type $+$. The proof for type $-$ is similar, and omitted.
\end{proof}

\section{Entries of matrices that involve $F^{0},F^{+},F^{-}$}
\label{entry}
Pick $x,y\in X$ such that $1<\partial(x,y)<k$. Pick $w\in \mathcal{A}_{xy}$. In this section we find the $(w,x)$-entries of many matrices that involve $F^{0},F^{+},F^{-}$.

\begin{lemma}
\label{F+F-}
    Pick $x,y\in X$ such that $1<\partial(x,y)<k$. For $w\in \mathcal{A}_{xy}$, the $(w,x)$-entry of $F^{+}F^{-}$ is equal to $0$.
\end{lemma}

\begin{proof}
    We consider three cases: (i) $w\in \mathcal{A}_{xy}^{0}$, (ii) $w\in \mathcal{A}_{xy}^{+}$, (iii) $w\in \mathcal{A}_{xy}^{-}$. We first assume case (i). Consider the right-hand side of (\ref{LR}). Referring to the table in Lemma \ref{structure}, the $(\mathcal{A}_{xy}^{0},\mathcal{B}_{xy})$-entry is equal to $0$. Hence, 
    \begin{equation*}
        (LR)_{w,x}=0.
    \end{equation*} 
    By Definition \ref{kdef} and Remark \ref{remark2}, the $(w,x)$-entry of 
    \begin{align*}
        &-(q-1)^{-1}\Bigl(\bigl(q^{\frac{n}{2}+1}K_1^{-1}K_2-q^{\frac{n-k}{2}+1}K_2\bigr)F^{+}+\bigl(q^{\frac{n}{2}+1}K_1K_2^{-1}-q^{\frac{k}{2}+1}K_1\bigr)F^{-}\Bigr)\\
            &\;\qquad \qquad \qquad -(q-1)^{-2}\Bigl(q^{\frac{n}{2}+1}K_1K_2-q^{n-\frac{k}{2}+1}K_1-q^{\frac{n+k}{2}+1}K_2+q^{n+1}I\Bigr)
    \end{align*}
    is equal to $0$. The result follows from the above comments and (\ref{LR}).

    We have proved the result for case (i). Next we assume case (ii). Consider the right-hand side of (\ref{LR}). Referring to the table in Lemma \ref{structure}, the $(\mathcal{A}_{xy}^{+},\mathcal{B}_{xy})$-entry is equal to $q^{i+1}[k-i]$. Hence, 
    \begin{equation*}
        (LR)_{w,x}=q^{i+1}[k-i].
    \end{equation*}
    By Definition \ref{kdef} and Remark \ref{remark2}, the $(w,x)$-entry of 
    \begin{equation*}
        \bigl(q^{\frac{n}{2}+1}K_1^{-1}K_2-q^{\frac{n-k}{2}+1}K_2\bigr)F^{+}
    \end{equation*}
    is equal to $(q-1)q^{i+1}[k-i]$, and the $(w,x)$-entries of 
    \begin{equation*}
        \bigl(q^{\frac{n}{2}+1}K_1K_2^{-1}-q^{\frac{k}{2}+1}K_1\bigr)F^{-}, \qquad \qquad (q-1)^{-2}\Bigl(q^{\frac{n}{2}+1}K_1K_2-q^{n-\frac{k}{2}+1}K_1-q^{\frac{n+k}{2}+1}K_2+q^{n+1}I\Bigr)
    \end{equation*}
    are both equal to $0$. The result follows from the above comments and (\ref{LR}). 

    We have proved the result for case (ii). For case (iii), the proof is similar, and omitted.
\end{proof}

\begin{lemma}
\label{F+F-var}
    Pick $x,y\in X$ such that $1<\partial(x,y)<k$. For $w\in \mathcal{A}_{xy}$, the following (i)--(vii) hold:
    \begin{enumerate}[label=(\roman*)]
        \item the $(w,x)$-entry of $F^{-}F^{+}$ is equal to $0$;
        \item if $w\in \mathcal{A}_{xy}^{+}$, then the $(w,x)$-entry of $F^{0}F^{-}$ is equal to $0$;
        \item if $w\in \mathcal{A}_{xy}^{+}$, then the $(w,x)$-entry of $F^{-}F^{0}$ is equal to $0$;
        \item if $w\in \mathcal{A}_{xy}^{+}$, then the $(w,x)$-entry of $\bigl(F^{-}\bigr)^2$ is equal to $0$;
        \item if $w\in \mathcal{A}_{xy}^{-}$, then the $(w,x)$-entry of $F^{0}F^{+}$ is equal to $0$;
        \item if $w\in \mathcal{A}_{xy}^{-}$, then the $(w,x)$-entry of $F^{+}F^{0}$ is equal to $0$;
        \item if $w\in \mathcal{A}_{xy}^{-}$, then the $(w,x)$-entry of $\bigl(F^{+}\bigr)^2$ is equal to $0$.
    \end{enumerate}
\end{lemma}

\begin{proof}
    (i) Combine Lemmas \ref{commute}, \ref{F+F-}.
    
    (ii) Since $w\in \mathcal{A}^{+}_{xy}$, we have $1<\partial(w,y)<k$. Since the entries of $F^{0}, F^{-}$ are nonnegative, we have $(F^0F^-)_{w,x}\geq 0$. Assume that the $(F^0F^-)_{w,x}>0$. Then $\mathcal{A}_{xy}^{-}\cap \mathcal{A}^{0}_{wy}$ is nonempty. Let $z\in \mathcal{A}_{xy}^{-}\cap \mathcal{A}^{0}_{wy}$. Note that $1<\partial(z,y)<k$ and $w\in \mathcal{A}_{zy}$. Since $(F^+)_{w,x}=1$ and $(F^-)_{x,z}=1$, we have $(F^{+}F^{-})_{w,z}>0$. This contradicts Lemma \ref{F+F-}. The result follows.

    (iii) Combine (ii) and Lemma \ref{commute}.

    (iv),(v),(vii) Similar to (ii) above.
    
    (vi) Combine (v) and Lemma \ref{commute}.
\end{proof}

\begin{lemma}
\label{0+}
    Pick $x,y\in X$ such that $1<\partial(x,y)<k$. Write $i=\partial(x,y)$. For $w\in \mathcal{A}^{0}_{xy}$ and $w'\in \mathcal{A}^{+}_{xy}$, the following (\ref{00+}), (\ref{0++}) hold:
        \begin{align}
            (q-1)[i]\bigl(F^0F^+\bigr)_{w,x}=q^{i+1}[n-k-i]\Bigl(\bigl(F^0\bigr)^2\Bigr)_{w',x};\label{00+}\\
            (q-1)[i]\Bigl(\bigl(F^+\bigr)^2\Bigr)_{w,x}=q^{i+1}[n-k-i]\bigl(F^0F^+\bigr)_{w',x}.
        \label{0++}
        \end{align}
\end{lemma}

\begin{proof}
    We first prove (\ref{00+}). Observe that $\bigl(F^{0}F^+\bigr)_{w,x}$ counts the number of vertices $z\in \mathcal{A}^{+}_{xy}$ adjacent to $w$ such that the edge $wz$ has type $0$. Hence, the total number of edges of type $0$ between $\mathcal{A}_{xy}^{0}$ and $\mathcal{A}_{xy}^{+}$ is equal to $a_i^{0}\bigl(F^0F^+\bigr)_{w,x}$. Observe that $\Bigl(\bigl(F^{0}\bigr)^2\Bigr)_{w',x}$ counts the number of vertices $z\in \mathcal{A}^{0}_{xy}$ adjacent to $w'$ such that the edge $w'z$ has type $0$. Hence, the total number of edges of type $0$ between $\mathcal{A}_{xy}^{0}$ and $\mathcal{A}_{xy}^{+}$ is equal to $a_i^{+}\Bigl(\bigl(F^0\bigr)^2\Bigr)_{w',x}$. Therefore, 
    \begin{equation*}
        a_i^{0}\bigl(F^0F^+\bigr)_{w,x}=a_i^{+}\Bigl(\bigl(F^0\bigr)^2\Bigr)_{w',x}.
    \end{equation*}
    Substitute the values of $a_i^{0}$ and $a_{i}^{+}$ from Lemma \ref{size} and simplify. The result follows.

    We have proved (\ref{00+}). The proof of (\ref{0++}) is similar, and omitted.
\end{proof}

\begin{lemma}
\label{0+entries}
    Pick $x,y\in X$ such that $1<\partial(x,y)<k$. Write $i=\partial(x,y)$. For $w\in \mathcal{A}^{0}_{xy}$ and $w'\in \mathcal{A}^{+}_{xy}$, 
    \begin{equation}
    \label{eq1}
        \bigl(F^{0}F^{+}\bigr)_{w,x}=\bigl(F^{+}F^{0}\bigr)_{w,x}=\Bigl(\bigl(F^0\bigr)^2\Bigr)_{w',x}=0, 
    \end{equation}
    \begin{equation}
    \label{eq2}
        \Bigl(\bigl(F^{+}\bigr)^2\Bigr)_{w,x}=q^{i+1}[n-k-i], \qquad \qquad \bigl(F^{0}F^{+}\bigr)_{w',x}=\bigl(F^{+}F^{0}\bigr)_{w',x}=(q-1)[i].
    \end{equation}
\end{lemma}

\begin{proof}
    Write 
    \begin{equation}
        \label{alphabeta}
        \alpha=\bigl(F^0F^+\bigr)_{w,x}, \qquad \qquad \beta=\Bigl(\bigl(F^+\bigr)^2\Bigr)_{w,x}.
    \end{equation}
    By Lemma \ref{commute}, $\bigl(F^+F^0\bigr)_{w,x}=\alpha$. By (\ref{00+}), 
    \begin{equation}
    \label{f00}
        \Bigl(\bigl(F^0\bigr)^2\Bigr)_{w',x}=\frac{q^{i+1}[n-k-i]}{(q-1)[i]}\alpha.
    \end{equation}
    By Lemma \ref{commute} and (\ref{0++}),
    \begin{equation*}
        \bigl(F^{0}F^{+}\bigr)_{w',x}=\bigl(F^{+}F^{0}\bigr)_{w',x}=\frac{(q-1)[i]}{q^{i+1}[n-k-i]}\beta.
    \end{equation*}
    It suffices to show that $\alpha=0$ and $\beta=q^{i+1}[n-k-i]$. By (\ref{fdef}),
    \begin{align}    
    \bigl(F^{0}\bigr)^2+F^{+}F^{0}+F^{-}F^{0}&=FF^{0},    \label{ff0eq}\\
    F^{0}F^{+}+\bigl(F^{+}\bigr)^2+F^{-}F^{+}&=FF^{+}.    \label{ff+eq}
    \end{align}
    We now calculate the $(w',x)$-entry of each term in (\ref{ff0eq}). By Lemma \ref{commute} and (\ref{0++}), 
    \begin{equation}
        \label{f+0}
        \bigl(F^+F^{0}\bigr)_{w'x}=\frac{(q-1)[i]}{q^{i+1}[n-k-i]}\beta.
    \end{equation}
    
    By Lemma \ref{F+F-var}(iii),
    \begin{equation}
    \label{f-0}
        \bigl(F^{-}F^{0}\bigr)_{w'x}=0.
    \end{equation}
    Observe that $\bigl(FF^0\bigr)_{w'x}$ counts the number of vertices $z\in F^0$ that are adjacent to $w'$. By the $(\mathcal{A}_{xy}^{+}, \mathcal{A}_{xy}^{0})$-entry of the table in Lemma \ref{structure}, 
    \begin{equation}
    \label{ff0}
        \bigl(FF^{0}\bigr)_{w',x}=(q-1)[i].
    \end{equation}
    Combine (\ref{f00}), (\ref{ff0eq}), (\ref{f+0})--(\ref{ff0}) to obtain
    \begin{equation}
    \label{system1}
        \frac{q^{i+1}[n-k-i]}{(q-1)[i]}\alpha+\frac{(q-1)[i]}{q^{i+1}[n-k-i]}\beta=(q-1)[i].
    \end{equation}
    Next we calculate the $(w,x)$-entry of each term in (\ref{ff+eq}). Observe that $\bigl(FF^+\bigr)_{w,x}$ counts the number of vertices $z\in F^{+}$ that are adjacent to $w$. Combining (\ref{alphabeta}), Lemma \ref{F+F-var}(i) and the $(\mathcal{A}^{0}_{xy},\mathcal{A}^{+}_{xy})$-entry of the table in Lemma \ref{structure}, we obtain
    \begin{equation}
    \label{system2}
        \alpha+\beta=q^{i+1}[n-k-i].
    \end{equation}
    Solve the system of equations (\ref{system1}), (\ref{system2}) to obtain the result.
\end{proof}

\begin{lemma}
\label{0-}
    Pick $x,y\in X$ such that $1<\partial(x,y)<k$. Write $i=\partial(x,y)$. For $w\in \mathcal{A}^{0}_{xy}$ and $w'\in \mathcal{A}^{-}_{xy}$, the following (\ref{00-}), (\ref{0--}) hold:
    \begin{align}
            (q-1)[i]\bigl(F^0F^-\bigr)_{w,x}=q^{i+1}[k-i]\Bigl(\bigl(F^0\bigr)^2\Bigr)_{w',x};        \label{00-}\\
            (q-1)[i]\Bigl(\bigl(F^-\bigr)^2\Bigr)_{w,x}=q^{i+1}[k-i]\bigl(F^0F^-\bigr)_{w',x}.\label{0--}
        \end{align}
\end{lemma}

\begin{proof}
    We first prove (\ref{00-}). Observe that $\bigl(F^{0}F^-\bigr)_{w,x}$ counts the number of vertices $z\in \mathcal{A}^{-}_{xy}$ adjacent to $w$ such that the edge $wz$ has type $0$. Hence, the total number of edges of type $0$ between $\mathcal{A}_{xy}^{0}$ and $\mathcal{A}_{xy}^{-}$ is equal to $a_i^{0}\bigl(F^0F^-\bigr)_{w,x}$. Observe that $\Bigl(\bigl(F^{0}\bigr)^2\Bigr)_{w',x}$ counts the number of vertices $z\in \mathcal{A}^{0}_{xy}$ adjacent to $w'$ such that the edge $w'z$ has type $0$. Hence, the total number of edges of type $0$ between $\mathcal{A}_{xy}^{0}$ and $\mathcal{A}_{xy}^{-}$ is equal to $a_i^{-}\Bigl(\bigl(F^0\bigr)^2\Bigr)_{w',x}$. Therefore, 
    \begin{equation*}
        a_i^{0}\bigl(F^0F^-\bigr)_{w,x}=a_i^{-}\Bigl(\bigl(F^0\bigr)^2\Bigr)_{w',x}.
    \end{equation*}
    Substitute the values of $a_i^{0}$ and $a_{i}^{-}$ from Lemma \ref{size} and simplify. The result follows.

    We have proved (\ref{00-}). The proof of (\ref{0--}) is similar, and omitted.
\end{proof}

\begin{lemma}
\label{0-entries}
    Pick $x,y\in X$ such that $1<\partial(x,y)<k$. Write $i=\partial(x,y)$. For $w\in \mathcal{A}^{0}_{xy}$ and $w'\in \mathcal{A}^{-}_{xy}$, 
    \begin{equation}
    \label{eq3}
        \bigl(F^{0}F^{-}\bigr)_{w,x}=\bigl(F^{-}F^{0}\bigr)_{w,x}=\Bigl(\bigl(F^0\bigr)^2\Bigr)_{w',x}=0, 
    \end{equation}
    \begin{equation}
    \label{eq4}
        \Bigl(\bigl(F^{-}\bigr)^2\Bigr)_{w,x}=q^{i+1}[k-i], \qquad \qquad \bigl(F^{0}F^{-}\bigr)_{w',x}=\bigl(F^{-}F^{0}\bigr)_{w',x}=(q-1)[i].
    \end{equation}
\end{lemma}

\begin{proof}
    Write 
    \begin{equation}
    \label{gammadelta}
        \gamma=\bigl(F^0F^-\bigr)_{w,x}, \qquad \qquad \delta=\Bigl(\bigl(F^-\bigr)^2\Bigr)_{w,x}.
    \end{equation}
    By Lemma \ref{commute}, $\bigl(F^-F^0\bigr)_{w,x}=\gamma$. By (\ref{00-}), 
    \begin{equation}
    \label{f00''}
        \Bigl(\bigl(F^0\bigr)^2\Bigr)_{w',x}=\frac{q^{i+1}[k-i]}{(q-1)[i]}\gamma.
    \end{equation}
    By Lemma \ref{commute} and (\ref{0--}),
    \begin{equation}
    \label{f0-v2}
        \bigl(F^{0}F^{-}\bigr)_{w',x}=\bigl(F^{-}F^{0}\bigr)_{w',x}=\frac{(q-1)[i]}{q^{i+1}[k-i]}\delta.
    \end{equation}
    It suffices to show that $\gamma=0$ and $\delta=q^{i+1}[k-i]$. We calculate the $(w',x)$-entry of each term in (\ref{ff0eq}). By Lemma \ref{F+F-var}(vi),
    \begin{equation}
    \label{f+0v2}
        \bigl(F^{+}F^{0}\bigr)_{w'x}=0.
    \end{equation}
    Observe that $\bigl(FF^{0}\bigr)_{w',x}$ counts the number of vertices $z\in \mathcal{A}_{xy}^{0}$ that are adjacent to $w'$. By the $(\mathcal{A}_{xy}^{-}, \mathcal{A}_{xy}^{0})$-entry of the table in Lemma \ref{structure}, 
    \begin{equation}
    \label{ff0v2}
        \bigl(FF^{0}\bigr)_{w',x}=(q-1)[i].
    \end{equation}
    Combine (\ref{ff0eq}), (\ref{f00''})--(\ref{ff0v2}) to obtain
    \begin{equation}
    \label{system3}
        \frac{q^{i+1}[k-i]}{(q-1)[i]}\gamma+\frac{(q-1)[i]}{q^{i+1}[k-i]}\delta=(q-1)[i].
    \end{equation}
    
    By (\ref{fdef}),
    \begin{equation}    
    \label{ff-eq}
    F^{0}F^{-}+F^{+}F^{-}+(F^{-})^2=FF^{-}.    
    \end{equation}
    
    We calculate the $(w,x)$-entry of each term in (\ref{ff-eq}). Observe that $\bigl(FF^{-}\bigr)_{w,x}$ counts the number of vertices $z\in \mathcal{A}_{xy}^{-}$ that are adjacent to $w$. Combining (\ref{gammadelta}), Lemma \ref{F+F-} and the $(\mathcal{A}^{0}_{xy},\mathcal{A}^{-}_{xy})$-entry of the table in Lemma \ref{structure}, we obtain
    \begin{equation}
    \label{system4}
        \gamma+\delta=q^{i+1}[k-i].
    \end{equation}
    Solve the system of equations (\ref{system3}), (\ref{system4}) to obtain the result.
\end{proof}

\begin{lemma}
\label{entries}
    Pick $x,y\in X$ such that $1<\partial(x,y)<k$. Write $i=\partial(x,y)$. Pick $w\in \mathcal{A}_{xy}$. The following (i)--(iii) hold.
    \begin{enumerate}[label=(\roman*)]
        \item Assume that $w\in \mathcal{A}_{xy}^{0}$. Then
        \begin{equation*}
            \Bigl(\bigl(F^0\bigr)^2\Bigr)_{w,x}=2q^{i}-q-2.
        \end{equation*}
        \item Assume that $w\in \mathcal{A}_{xy}^{+}$. Then
        \begin{equation*}
            \Bigl(\bigl(F^+\bigr)^2\Bigr)_{w,x}=q[n-k]-q^i-q.
        \end{equation*}
        \item  Assume that $w\in \mathcal{A}_{xy}^{-}$. Then
        \begin{equation*}
            \Bigl(\bigl(F^-\bigr)^2\Bigr)_{w,x}=q[k]-q^i-q.
        \end{equation*}
    \end{enumerate}
\end{lemma}

\begin{proof}
    We first prove (i). Pick $w\in \mathcal{A}^{0}_{xy}$. Recall $\alpha,\gamma$ from (\ref{alphabeta}), (\ref{gammadelta}). Observe that $\Bigl(\bigl(FF^0\bigr)^2\Bigr)_{w,x}$ counts the number of vertices $z\in \mathcal{A}^{0}_{xy}$ adjacent to $w$. Evaluate the $(w,x)$-entries for both sides of (\ref{ff0eq}) using the left equations of (\ref{alphabeta}), (\ref{gammadelta}) and the $(\mathcal{A}^{0}_{xy},\mathcal{A}^{0}_{xy})$-entry of the table in Lemma \ref{structure}; we get 
    \begin{equation*}
        \Bigl(\bigl(F^0\bigr)^2\Bigr)_{w,x}+\alpha+\gamma=2q^i-q-2.
    \end{equation*}
    The result follows from the fact that $\alpha=0$ and $\gamma=0$.

    We have proved (i). Proofs for (ii), (iii) are similar, and omitted.
    \end{proof}
    We summarize the results of this section in the theorem below.

    \begin{theorem}
    \label{entrytheorem}
        Pick $x,y\in X$ such that $1<\partial(x,y)<k$. Pick $w\in \mathcal{A}_{xy}$. Referring to the table below, for each matrix $M$ in the header column and each orbit $\mathcal{O}$ in the header row, the $(M,\mathcal{O})$-entry gives $M_{w,x}$ when $w\in \mathcal{O}$. Write $i=\partial(x,y)$.

        \begin{center}
            \begin{tabular}{c|c c c}
                & $\mathcal{A}^{0}_{xy}$ & $\mathcal{A}^{+}_{xy}$ & $\mathcal{A}^{-}_{xy}$ \\
                \hline \\
                $\bigl(F^{0}\bigr)^2$ & $2q^i-q-2$ & $0$ & $0$ \\ \\
                $F^0F^+$ & $0$ & $(q-1)[i]$ & $0$ \\ \\
                $F^0F^-$ & $0$ & $0$ & $(q-1)[i]$ \\ \\
                $F^+F^0$ & $0$ & $(q-1)[i]$ & $0$ \\ \\
                $\bigl(F^+\bigr)^2$ & $q^{i+1}[n-k-i]$ & $q[n-k]-q^i-q$ & $0$ \\ \\
                $F^+F^-$ & $0$ & $0$ & $0$\\ \\
                $F^-F^0$ & $0$ & $0$ & $(q-1)[i]$ \\ \\
                $F^-F^+$ & $0$ & $0$ & $0$ \\ \\
                $\bigl(F^-\bigr)^2$ & $q^{i+1}[k-i]$ & $0$ & $q[k]-q^i-q$
            \end{tabular}
        \end{center}
    \end{theorem}
    \begin{proof}
        Combine Lemmas \ref{F+F-}, \ref{F+F-var}, \ref{0+entries}, \ref{0-entries}, \ref{entries}.
    \end{proof}
\section{More results on the edges of different types in $\Gamma(x)$}
\label{typeII}
Pick $x,y\in X$ such that $1<\partial(x,y)<k$. Pick a pair of orbits $\mathcal{O}\in \{\mathcal{A}_{xy}^{0}, \mathcal{A}_{xy}^{+}, \mathcal{A}_{xy}^{-}\}$ and $\mathcal{N}$ from (\ref{5orbits}). Pick $w\in \mathcal{O}$. In this section we find the number of vertices $z\in \mathcal{N}$ adjacent to $w$ such that the edge $wz$ has (i) type $0$, (ii) type $+$, (iii) type $-$.

\begin{lemma}
    Pick $x,y\in X$ such that $1<\partial(x,y)<k$. For $w\in \mathcal{A}_{xy}$ and $z\in \mathcal{B}_{xy}\cup \mathcal{C}_{xy}$ that are adjacent, the edge $wz$ does not have type $0$, $+$, nor $-$. 
\end{lemma}
\begin{proof}
    Since $w\in \mathcal{A}_{xy}$, we have $\partial(w,y)=\partial(x,y)$. Since $z\in \mathcal{B}_{xy}\cup \mathcal{C}_{xy}$, we have $\partial(z,y)\neq \partial(x,y)$. Hence, $\partial(w,y)\neq \partial(z,y)$. The result follows.
\end{proof}

\begin{theorem}
\label{II1}
    Pick $x,y\in X$ such that $1<\partial(x,y)<k$. Pick $w\in \mathcal{A}^{0}_{xy}$. Referring to the table below, for each orbit $\mathcal{N}$ in the header row, and each type $t$ in the header column, the $(t,\mathcal{N})$-entry gives the number of vertices $z\in \mathcal{N}$ adjacent to $w$ such that $wz$ has the type $t$. Write $i=\partial(x,y)$.
    
    \begin{center}
    \begin{tabular}{c|c c c}
        Type & $\mathcal{A}^{0}_{xy}$ & $\mathcal{A}^{+}_{xy}$ & $\mathcal{A}^{-}_{xy}$\\
        \hline \\
        $0$ & $2q^i-q-2$ & $0$ & $0$\\ \\ 
        $+$ & $0$ & $q^{i+1}[n-k-i]$ & $0$\\ \\
        $-$ & $0$ & $0$ & $q^{i+1}[k-i]$
    \end{tabular}
    \end{center}
\end{theorem}

\begin{proof}
    Immediate from the $\mathcal{A}_{xy}^{0}$-column of the table in Theorem \ref{entrytheorem}.
\end{proof}

\begin{theorem}
\label{III1}
    Pick $x,y\in X$ such that $1<\partial(x,y)<k$. Pick $w\in \mathcal{A}^{+}_{xy}$. Referring to the table below, for each orbit $\mathcal{N}$ in the header row, and each type $t$ in the header column, the $(t,\mathcal{N})$-entry gives the number of vertices $z\in \mathcal{N}$ adjacent to $w$ such that $wz$ has the type $t$. Write $i=\partial(x,y)$.
    
    \begin{center}
    \begin{tabular}{c|c c c}
        Type & $\mathcal{A}^{0}_{xy}$ & $\mathcal{A}^{+}_{xy}$ & $\mathcal{A}^{-}_{xy}$\\
        \hline \\
        $0$ & $0$ & $(q-1)[i]$ & $0$ \\ \\ 
        $+$ & $(q-1)[i]$ & $q[n-k]-q^i-q$ & $0$\\ \\
        $-$ & $0$ & $0$ & $0$ 
    \end{tabular}
    \end{center}
\end{theorem}
\begin{proof}
    Immediate from the $\mathcal{A}_{xy}^{+}$-column of the table in Theorem \ref{entrytheorem}.
\end{proof}

\begin{theorem}
\label{III2}
    Pick $x,y\in X$ such that $1<\partial(x,y)<k$. Pick $w\in \mathcal{A}^{-}_{xy}$. Referring to the table below, for each orbit $\mathcal{N}$ in the header row, and each type $t$ in the header column, the $(t,\mathcal{N})$-entry gives the number of vertices $z\in \mathcal{N}$ adjacent to $w$ such that $wz$ has the type $t$. Write $i=\partial(x,y)$.
    
    \begin{center}
    \begin{tabular}{c|c c c}
        Type & $\mathcal{A}^{0}_{xy}$ & $
        \mathcal{A}^{+}_{xy}$ & $\mathcal{A}^{-}_{xy}$\\
        \hline \\
        $0$ & $0$ & $0$& $(q-1)[i]$\\ \\ 
        $+$ & $0$ & $0$ & $0$\\ \\
        $-$ & $(q-1)[i]$ & $0$ 
        & $q[k]-q^i-q$
    \end{tabular}
    \end{center}
\end{theorem}
\begin{proof}
    Immediate from the $\mathcal{A}_{xy}^{-}$-column of the table in Theorem \ref{entrytheorem}.
\end{proof}

We summarize the results from Sections \ref{types}, \ref{typeII} in the table below. Pick $x,y\in X$ such that $1<\partial(x,y)<k$. For each orbit $\mathcal{O}$ in the header column, and each orbit $\mathcal{N}$ in the header row, the $(\mathcal{O},\mathcal{N})$-entry gives a $3\times 1$-matrix that satisfies the following: each entry of the matrix corresponds to the number of vertices in $\mathcal{N}$ that share an edge of type $0$, $+$, $-$ respectively with a given vertex in $\mathcal{O}$. We omit the entries $(0,0,0)^{\top}$. Write $i=\partial(x,y)$.
    \begin{center}
\begin{tabular}{c|c c c c c}
     & $\mathcal{B}_{xy}$ & $\mathcal{C}_{xy}$ & $\mathcal{A}_{xy}^{0}$ & $\mathcal{A}_{xy}^{+}$ & $\mathcal{A}_{xy}^{-}$\\
    \hline \\
    $\mathcal{B}_{xy}$ & $\begin{pmatrix}
        \substack{2q^{i+1}-q-1}\\
        \substack{q^{i+2}[n-k-i-1]}\\
        \substack{q^{i+2}[k-i-1]}
    \end{pmatrix}$ &  &  &  & \\
   \\
    $\mathcal{C}_{xy}$ &  & $\begin{pmatrix}
        0\\
        q[i-1]\\
        q[i-1]
    \end{pmatrix}$ &  &  & \\
   \\
    $\mathcal{A}_{xy}^{0}$ &  &  & $\begin{pmatrix}
        2q^{i}-q-2\\0\\0
    \end{pmatrix}$ & $\begin{pmatrix}
        0\\q^{i+1}[n-k-i]\\0
    \end{pmatrix}$ & $\begin{pmatrix}
        0\\0\\q^{i+1}[k-i]
    \end{pmatrix}$ \\
    \\
    $\mathcal{A}_{xy}^{+}$ & & & $\begin{pmatrix}
        0\\(q-1)[i]\\0
    \end{pmatrix}$ & $\begin{pmatrix}
       (q-1)[i]\\ q[n-k]-q^i-q \\0
    \end{pmatrix}$ & \\
    \\
    $\mathcal{A}_{xy}^{-}$ & & & $\begin{pmatrix}
        0\\0\\(q-1)[i]
    \end{pmatrix}$ & & $\begin{pmatrix}
       (q-1)[i] \\0\\ q[k]-q^i-q
    \end{pmatrix}$\\
\end{tabular}
\end{center}

\section{Appendix}
\label{appendix}
In this section, we recall from \cite{Watanabe} the relations between the generators of $\mathcal{H}$.

\begin{lemma}{\rm{\cite[Lemma~7.4]{Watanabe}}}
\label{appendix1}
    The following (i)--(viii) hold:
    \begin{enumerate}[label=(\roman*)]
        \item $K_1L_1=qL_1K_1$;
        \item $K_1L_2=L_2K_1$;
        \item $qK_1R_1=R_1K_1$;
        \item $K_1R_2=R_2K_1$;
        \item $K_2L_1=L_1K_2$;
        \item $qK_2L_2=L_2K_2$;
        \item $K_2R_1=R_1K_2$;
        \item $K_2R_2=qR_2K_2$.
    \end{enumerate}
\end{lemma}

\begin{lemma}{\rm{\cite[Lemma~7.5]{Watanabe}}}
\label{appendix2}
    The following (i)--(iv) hold:
    \begin{enumerate}[label=(\roman*)]
        \item $L_1R_2=R_2L_1$;
        \item $L_2R_1=R_1L_2$;
        \item $qL_1L_2=L_2L_1$;
        \item $R_1R_2=qR_2R_1$.
    \end{enumerate}
\end{lemma}

\begin{lemma}{\rm{\cite[Lemma~7.6]{Watanabe}}}
\label{appendix3}
    The following (i)--(iv) hold.
    \begin{enumerate}[label=(\roman*)]
        \item $R_1^2L_1-(q+1)R_1L_1R_1+qL_1R_1^2=-q^{\frac{n}{2}-1}(q+1)K_1^{-1}K_2R_1$;
        \item $qR_2^2L_2-(q+1)R_2L_2R_2+L_2R_2^2=-q^{\frac{n}{2}}(q+1)K_1K_2^{-1}R_2$;
        \item $qL_1^2R_1-(q+1)L_1R_1L_1+R_1L_1^2=-q^{\frac{n}{2}}(q+1)K_1^{-1}K_2L_1$;
        \item $L_2^2R_2-(q+1)L_2R_2L_2+qR_2L_2^2=-q^{\frac{n}{2}-1}(q+1)K_1K_2^{-1}L_2$.
    \end{enumerate}
\end{lemma}

\begin{lemma}{\rm{\cite[Lemma~7.7]{Watanabe}}}
\label{appendix4}
    The generators of $\mathcal{H}$ satisfy
    \begin{equation*}
        L_1R_1-R_1L_1+L_2R_2-R_2L_2=q^{\frac{n}{2}}(q-1)^{-1}(K_1K_2^{-1}-K_1^{-1}K_2).
    \end{equation*}
\end{lemma}

\section*{Acknowledgement}
The author would like to thank Paul Terwilliger for giving many valuable ideas and comments, especially the idea of using the algebra $\mathcal{H}$ to obtain many results in this paper. 

\section*{Declarations}
\subsection*{Data Availability Statement}
No datasets were generated or analyzed during the current study.

\subsection*{Conflict of interest}
The author has no relevant financial or non-financial interests to disclose.

\newpage
Ian Seong\\
Department of Mathematics\\
Williams College \\
18 Hoxsey St \\
Williamstown, MA 01267-2680 USA \\
email: is11@williams.edu\\
\end{document}